\documentclass[a4paper,12pt]{amsart}
\usepackage[margin=1.4in]{geometry} 
\usepackage{tikz}
\usepackage{amsfonts,amssymb,amsmath,amsthm}
\usepackage{enumerate}
\usepackage[shortlabels]{enumitem}
\usepackage{mathabx}             
\usepackage{xargs}               
\usepackage{xfrac} 
\usepackage{bm}    
\usepackage{comment}
\usepackage{xifthen}
\usepackage{graphicx}
\usepackage{float}
\usepackage{wrapfig}
\usepackage{caption}
\usepackage{epsfig}
\usepackage{chngcntr}
\usepackage[capitalize,noabbrev]{cleveref}
\usepackage{nicefrac}
\usepackage{thmtools}

\DeclareMathAlphabet{\mathpzc}{OT1}{pzc}{m}{it}

\DeclareMathOperator{\SO}{SO}

   \newcommand{\N}{\mathbb{N}}
   \newcommand{\Z}{\mathbb{Z}}
   
   \newcommand{\R}{\mathbb{R}}

\newcommand{\GmodGamma}{G / \Gamma}
\newcommand{\asymneighborhood}{\mathcal{S}_\mu}
\newcommand{\tOmega}{\widetilde{\Omega}}
\newcommand{\Hd}{\mathbb{H}^d}

\newcommand{\cupdot}{\mathbin{\mathaccent\cdot\cup}}
\newcommand{\FHd}{\mathcal{F}\Hd}
\newcommand{\acts}{\;\rotatebox[origin=c]{-90}{$\circlearrowright$}\;}
\newcommand{\PbackslashG}{{P \backslash G}}
\newcommand{\tmu}{\tilde{\mu}}
\newcommand{\injrad}{\operatorname{rad_{\mathrm{inj}}}}

\theoremstyle{plain}
\newtheorem{thm}{Theorem}[section]
\newtheorem*{thm*}{Theorem}
\newtheorem{lemma}[thm]{Lemma}
\newtheorem*{lemma*}{Lemma}
\newtheorem{prop}[thm]{Proposition}

\newtheorem*{claim*}{Claim}
\newtheorem{cor}[thm]{Corollary}

\theoremstyle{definition}
\newtheorem{dfn}{Definition}[section]

\theoremstyle{remark}
\newtheorem*{rem*}{Remark}
\newtheorem{rem}[thm]{Remark}

\newtheorem*{exmp*}{Example}

\numberwithin{equation}{section} 

\newtheoremstyle{TheoremNum}
{\topsep}{\topsep}
{\itshape}        
{}                
{\bfseries}       
{.}               
{ }               
{\thmname{#1}\thmnote{ \bfseries #3}}
\theoremstyle{TheoremNum}
\newtheorem{cornm}{Corollary}

\begin{document}
	
\graphicspath{ {} }

\pagestyle{plain}
\title{On Radon measures invariant under horospherical flows on geometrically infinite quotients}
\author{Or Landesberg and Elon Lindenstrauss}
\thanks{The authors acknowledge support from ERC 2020 grant HomDyn (grant no.~833423) and ISF
	grant 891/15. E.L.'s stay at the IAS during the fall of 2018 was supported in part by NSF grant DMS-1638352}

\begin{abstract}
	We consider a locally finite (Radon) measure on $ \SO^+(d,1)/ \Gamma $ invariant under a horospherical subgroup of $ \SO^+(d,1) $ where $ \Gamma $ is a discrete, but not necessarily geometrically finite, subgroup. 
	We show that whenever the measure does not observe any additional invariance properties then it must be supported on a set of points with geometrically degenerate trajectories under the corresponding contracting $ 1 $-parameter diagonalizable flow (geodesic flow). 
\end{abstract}

\maketitle

\section{Introduction and Statement of Results}

Consider a geometrically infinite discrete subgroup $ \Gamma $ of $ G=\SO^+(d,1) $, the group of orientation preserving isometries of the hyperbolic $ d $-space $\Hd$. We study the locally finite horospherical invariant measures on $ \GmodGamma $.

Classification of invariant measures under horospherical flows on hyperbolic manifolds has a long history beginning with Furstenberg's proof of unique ergodicity of the horocycle flow on compact hyperbolic surfaces \cite{Furstenberg1973}. More generally, finite measures on $G/\Gamma$ which are invariant and ergodic under the horospherical flow are either supported on periodic horospheres bounding a cusp or are equal to the volume measure (which therefore must be finite). This was shown by Dani \cite{Dani1978} and Veech \cite{Veech1977} for $G/\Gamma$ of finite volume; for general quotients this result is a special case of Ratner's measure classification theorem \cite{Ratner1991}. However, when considering $G/\Gamma$ with infinite volume the natural measure classification question is not classifying finite measures but rather that of classifying locally finite, a.k.a.\ Radon, measures. 

When considering Radon measures there is a clear distinction between geometrically finite and geometrically infinite discrete subgroups. The cusps of geometrically finite manifolds are somewhat more complicated than that of finite volume manifolds, as they may have non-maximal rank. We say an orbit of the horospherical flow on $G/\Gamma$ \emph{bounds a cusp} if its projection to the corresponding hyperbolic orbifold $\Hd/\Gamma$ is the boundary of a horoball emanating from one of the cusps of $\Hd/\Gamma$. Equivalently, a horospherical orbit bounds a cusp if its orbit under the one-parameter $\R$-split group contracting the horosphere tends to a parabolic limit point of $\Gamma$ --- see \S\ref{sec:regular cover}. This is the natural generalization of periodic orbits of a horosphere on finite volume quotients of $G$, and somewhat loosely we shall say a measure on $G/\Gamma$ is \emph{non-periodic} if it gives zero measure to the horospheres bounding a cusp.

Geometrically finite manifolds turn out to exhibit a unique recurrent and non-periodic horospherical invariant Radon measure. This result is due to Roblin \cite{Roblin2003} over the unit tangent bundle, extending earlier work by Burger \cite{Burger1990}, and due to Winter \cite{Winter2015} over the full frame bundle.

It was discovered by Babillot and Ledrappier \cite{Babillot-Ledrappier} that there is no such uniqueness phenomenon in the geometrically infinite setting. In particular, Babillot and Ledrappier showed that abelian covers of compact hyperbolic surfaces support an uncountable family of horocycle  invariant ergodic and recurrent Radon measures. Sarig \cite{Sarig2004} showed that the class of measures constructed by Babillot and Ledrappier are the only horocycle invariant Radon measures on abelian covers of compact surfaces.
Later Ledrappier and Sarig \cite{Ledrappier-Sarig} extended the measure classification result to all regular covers of finite volume surfaces (though in this case the classification one gets is somewhat less explicit, similar to the measure classification results we give here; see \S\ref{Subsection_measuredecomp}). Ledrappier \cite{Ledrappier2008} extended these results to the unit tangent bundle of normal covers of compact manifolds of variable negative curvature and arbitrary dimension. 
Oh and Pan \cite{Oh-Pan} recently strengthened Ledrappier's result to address the horospherical flow over the full frame bundle of abelian covers of compact hyperbolic manifolds of arbitrary dimension.

A key step in all the above is obtaining additional invariance properties of the given horospherically invariant measure, specifically proving that in the cases considered by these authors such measures (unless periodic) are quasi-invariant with respect to the geodesic flow. 

Let $ \{a_t\}_{t \in \R} $ denote the 1-parameter $\R$-diagonal group in $G$ and let $ U < G $ be the associated unstable horospherical subgroup. The flow given by $a_t$ on $ \GmodGamma $ projects to the geodesic flow on the unit tangent bundle of $\Hd/\Gamma$. Denote by $ N_G(U) $ the normalizer of $ U $ in $ G $. One consequence of our main result is the following:
\begin{cor}\label{cor_GF_intro}
	Let $ \Gamma_0  $ be a geometrically finite Zariski dense discrete subgroup of $ G $ and let $ \{e\} \neq \Gamma \lhd \Gamma_0 $, i.e. $ \GmodGamma $ is a regular cover of $ \GmodGamma_0 $. Let $ \mu $ be a $ U $-ergodic and invariant Radon measure on $ \GmodGamma $. Then one of the following three possibilities holds:
	\begin{enumerate}[leftmargin=*]
		\item $ \mu $ is supported on a wandering horosphere.
		\item $ \mu $ is supported on a lift to $ \GmodGamma $ of a horosphere in $ \GmodGamma_0 $ bounding a cusp.
		\item $ \mu $ is $ N_G(U) $-quasi-invariant.
	\end{enumerate}
\end{cor}

Note that this result is new even in the case of $d=2$. 

\medskip

More generally, we aim to understand what conditions on $ \Gamma $ or $ \mu $, a horospherically invariant Radon measure, imply $ N_G(U) $-quasi-invariance.
Inspired by Sarig's results on ``weakly tame'' surfaces \cite{Sarig2010}, we give a general condition for quasi-invariance depending on the geometric ``scenery'' encountered by geodesic rays $ \{ a_{-t}x \}_{t\geq 0} $ for $ \mu $-a.e.~$ x \in \GmodGamma $. 
This relation between the geometric scenery encountered and the invariance properties of $\mu$ is exemplified in the following statement, proven in \S\ref{Subsection_injrad}, which is a special case of \cref{main_theorem} below:
\begin{cor}\label{cor_injrad_intro}
	Let $ \Gamma < G $ be a discrete group with injectivity radius uniformly bounded away from 0. Let $ \mu $ be any $ U $-ergodic and invariant Radon measure on $ \GmodGamma $. Then at least one of the following holds:
	\begin{enumerate}[leftmargin=*]
		\item $ \mu $ is quasi-invariant under some hyperbolic element of $ N_G(U) $.
		\item $ \lim_{t \to \infty} \injrad(a_{-t} x) = \infty $ for $ \mu $-a.e.~$ x $.
	\end{enumerate}
	Where $\injrad (y) $ is the supremal injectivity radius at the point $ y \in \GmodGamma $.
\end{cor}

\medskip
\noindent
We note that this result is somewhat reminiscent to a result in the topological category by Maucourant and Schapira (\cite[Cor.~3.2]{Maucourant-Schapira}) classifying points in $G/\Gamma$ whose $U$-orbits are dense in the non-wandering set for~$U$.

In this paper we give a sufficient condition for $ N_G(U) $-quasi-invariance  applicable to horospherical flows over the frame bundle of a wide variety of  hyperbolic manifolds of arbitrary dimension.  

While heavily inspired by Ledrappier and Sarig, our proof uses a significantly different approach, one which does not rely on any symbolic representation of the dynamics at hand. In particular, we give a new proof of Sarig's quasi-invariance result for weakly tame surfaces \cite{Sarig2010}, as well as \cite{Ledrappier-Sarig}. Our starting point is an insight that we learned from Ratner's proof of the horocycle measure classification theorem in her expository paper \cite{Ratner1992_SL2(R)} which roughly states that whenever $ a_{-t} x \approx a_{-t} y $ then the respective ergodic averages at $ x $ and $ y $ of scale $ e^t $ satisfy
\[ e^{-t(d-1)}\int_{a_t B_1^U a_{-t}} f(ux) du \approx e^{-t(d-1)}\int_{a_t B_1^U a_{-t}} f(uy) du, \]
where $B_1^U$ denotes a unit ball around the identity $e$ in $U$.
We adapt this argument to suit the infinite measure setting and make crucial use of Hochman's multiparameter ratio ergodic theorem \cite{Hochman2010}. 

\subsection{Statement of the Main Results}

Let $ G=\SO^+(d,1) $ be the group of orientation preserving isometries of hyperbolic $ d $-space, equipped with a right-invariant metric. We shall make use of the following subgroups of $ G $:
\begin{itemize}[leftmargin=*]
	\item $ A=\{a_t\}_{t \in \R} $ --- the Cartan subgroup.
	\item $ U = \{ u : a_{-t}ua_{t} \to e \text{ as } t \to \infty \} $ --- the unstable horospherical subgroup.
	\item $ N $ --- the stable horospherical subgroup.
	\item $ K \cong SO(d) $ --- the maximal compact subgroup.
	\item $ M=Z_K(A) $ --- the centralizer of $ A $ in $ K $.
\end{itemize}

The normalizer of $ U $ in $ G $ is $ N_G(U) = MAU $. Given a discrete subgroup $ \Gamma < G $, the space $ \GmodGamma $ is identified with the frame bundle $ \FHd / \Gamma $. The left action of $ A $ on $ \GmodGamma $ corresponds to the geodesic frame flow on $ \FHd / \Gamma $ and the action of $ M $ corresponds to a rotation of the frames around the direction of the geodesic.

Given a $ U $-ergodic and invariant Radon measure (e.i.r.m.) $ \mu $ on $ \GmodGamma $ denote by $ H_\mu $ the stabilizer in $ MA $ of the measure class of $ \mu $, i.e.
\[ H_\mu = \text{Stab}_{MA}([\mu]) = \{ a' \in MA \ :\ a'.\mu \sim \mu \}. \]
Note that $ H_\mu $ is a closed subgroup of $ MA $.

\begin{dfn}\label{dfn_alpha}
	Given any hyperbolic element $ \gamma \in G $ define 
	\[ \alpha(\gamma) = a' \in MA \] 
	whenever the following two conditions hold:
	\begin{enumerate}
		\item $ u_0 \gamma = n_0 a' u_0 $ for some $ u_0 \in U $ and $ n_0 \in N $
		\item $ \|(a')^{-1} u a'\|_U < \|u\|_U \ \text{ for all } u \in U $
	\end{enumerate}
	where $ \|\cdot\|_U $ is some fixed bi-invariant metric on the group $ U $. Set $ \alpha(\gamma)=e $ whenever such $ a', n_0, u_0 $ do not exist or whenever $ \gamma $ is non-hyperbolic (i.e. elliptic or parabolic). 
\end{dfn}

The map $ \alpha $ is well defined and may be thought of as a map associating to a hyperbolic element $ \gamma $ a specific representative to the action of $ \gamma $ on its axis in the hyperbolic frame bundle $ \FHd $. See \S\ref{Section_reparam_map} for further discussion.

\medskip
We can now state our main theorem:
\begin{thm}\label{main_theorem}
	Let $ \Gamma < G $ be any discrete subgroup. Let $ \mu $ be a $ U $-e.i.r.m.~on $ \GmodGamma $, then 
	\[ \alpha\left(\bigcap_{n=1}^\infty \overline{\bigcup_{t \geq n} a_{-t}g\Gamma g^{-1}a_t}\right) \subseteq H_\mu \;\; \text{ for $ \mu $-a.e.~$ g\Gamma $}. \]
\end{thm}

\vspace{4mm}
Given $ x=g\Gamma \in \GmodGamma $ denote
\[ \mathcal{S}_{x} =\bigcap_{n=1}^\infty \overline{\bigcup_{t \geq n} a_{-t}g\Gamma g^{-1}a_t}. \]
Recall that $ a_{-t}g\Gamma g^{-1}a_t $ is the stabilizer of the point $ a_{-t}x $ in $ \GmodGamma $. Hence $ \mathcal{S}_{x} $ may be viewed as the set of accumulation points of elements of $ \Gamma $ as ``seen'' from the viewpoint of the geodesic ray $ \left\{a_{-t}x\right\}_{t\geq 0} $.

Note that the map $ x \mapsto \mathcal{S}_{x} $ is $ U $-invariant, i.e. $ \mathcal{S}_{u.x} =\mathcal{S}_{x} $ for any $ u \in U $. It is also measurable, as a map into the space of closed subsets of $ G $. Hence given an ergodic measure $ \mu $, there exists $ \mathcal{S}_{\mu} \subset G $ satisfying 
\[ \mathcal{S}_{x}=\mathcal{S}_{\mu} \;\; \text{for $ \mu $-a.e.~$ x \in \GmodGamma $.} \] 
Thus an equivalent formulation of \cref{main_theorem}:

\begin{thm*}[Alternative formulation of \cref{main_theorem}]
	Let $ \Gamma < G $ be any discrete subgroup. Let $ \mu $ be a $ U $-e.i.r.m.~on $ \GmodGamma $, then $ \alpha(\asymneighborhood) \subseteq H_\mu $.
	
\end{thm*}

We derive the following condition for $ MA $-quasi-invariance:
\begin{cor}\label{cor_MA-qi}
	Let $ \Gamma $ be any discrete subgroup of $ G $ and let $ \mu $ be any $ U $-e.i.r.m.~on $ \GmodGamma $. If $ \asymneighborhood $ contains a Zariski dense subgroup of $ G $ then $ \mu $ is $ MA $-quasi-invariant.	
\end{cor}

As in \cite{Babillot2004,Burger1990,Ledrappier-Sarig,Sarig2010}, such quasi-invariance implies a representation of the above invariant measures using $\Gamma$-conformal measures on the boundary $S^{d-1}$ of $\Hd$: any $ U $-e.i.r.m.~on $ \GmodGamma $ which is $ MA $-quasi-invariant lifts to a measure $ \tmu $ on $ G $ presented in $  S^{d-1}\times M \times A \times U  $ coordinates as: 
\[ d\tmu = e^{\beta t} d\nu dm dt du  \] 
where $ \nu $ is a $ \Gamma $-conformal measure with parameter $ \beta $ on $ S^{d-1} $; see \S\ref{Subsection_measuredecomp}.

\subsection{Structure of the Paper}
The paper is organized as follows --- \cref{Section_Special_Case} presents some of the main ideas of the proof in an illustrative special case. Sections \ref{Section_reparam_map} and \ref{Section_proof_mainthm} contain the proof of \cref{main_theorem}. In \cref{Section_corollaries} we deduce corollaries~\ref{cor_GF_intro}, \ref{cor_injrad_intro} and \ref{cor_MA-qi} from our main result.

\subsection{Acknowledgments}

The authors would like to thank Omri Sarig for helpful suggestions and insights and for his encouragement throughout the process of this study.

This paper is part of the first author's PhD thesis conducted at the Hebrew University of Jerusalem under the guidance of the second author. The first author is deeply grateful to Elon for introducing him to ergodic theory and homogeneous dynamics and for his continuous support and encouragement. The first author would also like to thank Ilya Khayutin and Manuel Luethi for many helpful and enjoyable conversations. The second author would like to thank Michael Larsen for helpful discussions. The second author would also like to thank the IAS for providing ideal working conditions in the fall of 2018. Finally, we would like to thank Hee Oh and an anonymous referee for helpful comments on an earlier version of this paper.

\section{Special case of \cref{main_theorem}}\label{Section_Special_Case}

An important ingredient in the proof of \cref{main_theorem} is the ratio ergodic theorem for non-singular actions of $ \R^n $, proven by Hopf for $ n=1 $ and by Hochman \cite{Hochman2010} for all $ n $:

\begin{thm}[Ratio Ergodic Theorem (Hochman)]\label{thm_ratio}
	Let $ \|\cdot \| $ be any norm on $ \R^n $ and let $ (u_s) $ be a free, non-singular and ergodic $ \R^n $-action on a standard $ \sigma $-finite measure space $ (X,\mu) $. Given any $ f,h \in L^1(X) $ with $ \int h d\mu \neq 0 $:
	\[ \frac{\int_{\|s\|<S} f(u_s x)ds}{\int_{\|s\|<S} h(u_s x)ds}  
	\underset{S\to \infty}{\longrightarrow} \frac{\int f d\mu}{\int h d\mu} \qquad \mu \text{-almost surely.} \]
\end{thm}

\begin{dfn}
	A point $ x \in \GmodGamma $ is called $ \mu $-generic if it satisfies the ratio ergodic theorem for $ \mu $ w.r.t all functions in $ C_c(\GmodGamma) $, i.e. if
	\[ \frac{\int_{\|u\|<S} f(ux) du}{\int_{\|u\|<S} g(ux) du} \underset{S\to \infty}{\longrightarrow} \frac{\int f d\mu}{\int g d\mu} \]
	for all $ f,g \in C_c(\GmodGamma) $ with $ \int g d\mu \neq 0 $.
\end{dfn}
\noindent Note that the separability of $ C_c(\GmodGamma) $ implies that $ \mu $-a.e.~$ x \in \GmodGamma $ is $ \mu $-generic.

\medskip
A particularly illuminating special case of our main result is the case when for $ \mu $-a.e.~$ x $ the geodesic ray $ (a_{-t}x)_{t>0} $ enters arbitrarily small neighborhoods of closed geodesics of some approximate fixed type $ a^\prime \in MA \smallsetminus M $. More precisely, whenever for $ \mu $-a.e.~$ x=g\Gamma $ there exist sequences $ t_n \to \infty $ and $ \gamma_n \in a_{-t_n}g\Gamma g^{-1}a_{t_n} $ satisfying $ \gamma_n \to a^\prime $.

\begin{figure}[t]	
	\begin{tikzpicture}[scale=1.25]
	\draw[thick] (1,5) .. controls (2,4.5) and (4,4.5) .. (5,5);
	\draw[thick] (1,3) .. controls (2,3.5) and (4,3.5) .. (5,3);
	
	\draw[black!40!blue,dashed] (0.1,1) -- (0.5,2);
	\draw[black!40!blue] (0.5,2) .. controls (1.17,3.6) .. (2,4.7);
	\draw[black!40!blue,dashed,opacity=0.5] (2,4.7) .. controls (2.25,4) .. (2.2,3.33);
	\draw[black!40!blue] (2.2,3.33) .. controls (2.4,4) .. (2.6,4.62);
	\draw[black!40!blue,dashed,opacity=0.5] (2.6,4.62) .. controls (2.7,4) .. (2.7,3.34);
	\draw[black!40!blue] (2.7,3.34) .. controls (2.9,4) .. (3.2,4.63);
	\draw[black!40!blue,dashed,opacity=0.5] (3.2,4.63) .. controls (3.6,4) .. (3.6,3.33);
	\draw[black!40!blue] (3.6,3.33) .. controls (4.4,4.2) .. (5.2,4.8);
	
	\draw[black!15!green,thick] (1.3,4.1) .. controls (2.6,3.97) and (3.1,3.97) .. (3.7,3.95);
	\draw[black!15!green,thick] (2.35,4.17) .. controls (2.7,4.15) and (3,4.11) .. (3.5,4.11);
	
	\draw[black!15!green,thick] (-3,0.2) .. controls (-1,0.1) and (1,0) .. (5.5,0.1);
	\draw[black!15!green,thick,dashed] (-3.5,0.23) .. controls (-3.3,0.205) and (-3.1,0.2) .. (-3,0.2);
	\draw[black!15!green,thick,dashed] (5.5,0.1) .. controls (5.6,0.1) and (5.8,0.105) .. (6,0.125);
	\draw[black!15!green,thick] (-2,0.3) .. controls (-0.2,0.2) and (3.5,0.17) .. (5.35,0.27);
	
	\draw (2.4,4) node {\tiny \textbullet};
	\draw (2.418,4.014) node[anchor=north east] {\footnotesize $ a_{-t}x $};
	\draw[->] (2.4,4) -- (2.3,3.7);
	
	\draw (2.96,4.1) node {\tiny \textbullet};
	\draw (2.96,4.04) node[anchor=south west] {\footnotesize $ a^\prime a_{-t}x $};
	\draw[->] (2.96,4.1) -- (2.85,3.8);
	
	\draw (-0.18,0.074) node {\tiny \textbullet};
	\draw (-0.25,0.03) node[anchor=north west] {\footnotesize $ x $};
	\draw[->] (-0.195,0.06) -- (-0.4,-0.5);
	
	\draw (1.802,0.198) node {\tiny \textbullet};
	\draw (1.8,0.18) node[anchor=south east] {\footnotesize $ a^\prime x $};
	\draw[->] (1.79,0.18) -- (1.65,-0.3);
	\end{tikzpicture}
	\caption{}\label{figure_special}
\end{figure}

At the heart of the proof is the following observation --- If the points $ a_{-t}x $ and $ a_{-t}y $ are close then the ergodic $ U $-averages of scale $Ce^t $ around $ x $ and $ y $ are also close (in an appropriate sense; see \cref{figure_special}). For any $ f \in C_c(\GmodGamma) $, taking $ y=a^\prime x $, arbitrarily large $ t $ and using the fact that $ a^\prime $ centralizes $ a_t $ forces the ergodic averages of $ f $ and $ f \circ a^\prime $ to be comparable, implying $ a^\prime $-quasi-invariance. 

\begin{prop}[Special case]\label{thm_special_case}
	Let $ \mu $ be a $ U $-e.i.r.m on $ \GmodGamma $. Assume that $ a^\prime \in \asymneighborhood \cap (MA \smallsetminus M) $ then $ a^\prime \in H_\mu $.
\end{prop}

\begin{lemma}\label{lma_witness}
	Let $ a^\prime \in MA $ and $ 0<\varepsilon $ with $ B_\varepsilon^{MA}(a^\prime) \cap H_\mu = \emptyset $. Then for any $ 0< \eta $ there exists a function $ 0\leq f \in C_c^1(\GmodGamma) $ with $ \int f d\mu >0 $ and
	\[ \frac{\int  f_\bot d\mu}{\int f d\mu} <\eta \]
	where
	\[ f_\bot(x)=\max \{ f(v \ell x) : v \in B^N_\delta(e),\; \ell \in B^{MA}_\varepsilon(a^\prime) \} \]
	for some $ \delta = \delta(\eta,f) $.
\end{lemma}

\begin{proof}
	Given any $ \mu $-generic point $ x \in \GmodGamma $ and any element $ \ell $ in the normalizer of $ U $, the point $ \ell x $ is $ \ell. \mu $-generic.
	Let $ K $ be any compact set of positive $ \mu $-measure consisting of $ \mu $-generic points. By the above observation we get that the set
	\[ K_\bot=\bigcup_{\ell \in B^{MA}_\varepsilon(a^\prime)}\ell^{-1} K \]
	satisfies $  K_\bot \cap K = \emptyset $.
	Taking $ \chi_K \leq f \leq 1 $ with $ \int (f-\chi_K)d\mu < \frac{\eta}{2} \mu(K) $ and taking $ 0<\delta $ sufficiently small ensures that the function
	\[ f_\bot=\max \{ f(v \ell x) : v \in B^N_\delta(e),\; \ell \in B^{MA}_\varepsilon(a^\prime) \} \]
	is very close to $ \chi_{K_\bot} $ and satisfies 
	$ \int f_\bot d\mu < \eta \mu(K) \leq \eta \int f d\mu $
	as required.
\end{proof}

\begin{proof}[Proof of \cref{thm_special_case}]
	The group $ H_\mu $ is closed in $ MA $ since
	\[ H_\mu=\bigcap_{f,g\in C_c(X)}\left \{ \ell \in MA : \ell.\mu(f) \cdot \mu(g) = \mu(f) \cdot \ell.\mu(g)  \right \}. \]
	Assume by contradiction that $ a^\prime \notin H_\mu $, then there exists an $ 0<\varepsilon $ for which $ B_\varepsilon^{MA}(a^\prime) \cap H_\mu = \emptyset $. The groups $ U $ and $ \R^{d-1} $ are isomorphic as Lie groups via a map $ s \mapsto u(s) $. Conjugation in $ U $ by $ a^\prime $ induces a similarity map $ kO $ on $ \R^{d-1} $, where $ O \in SO(d-1) $ and $ 0<k $. Both $ \asymneighborhood $ and $ H_\mu $ are closed under inversion hence we may assume without loss of generality that the conjugation map sending $u(s)$ to
	\[ (a^\prime)^{-1}u(s) a^\prime = u({kOs}) \]
	is contracting, i.e $ k<1 $. 
	
	Let $ f, f_\bot $ be the functions constructed in \cref{lma_witness} for $ \eta < \frac{k}{4} $ and $ \varepsilon $. Let $ x=g\Gamma $ be a $ \mu $-generic point with $ a^\prime \in \mathcal{S}_x $. By the definition of $ \mathcal{S}_x $ there exist sequences $ t_n \to \infty $ and $ \gamma_n \to a^\prime $ such that
	\begin{equation}\label{tech_specialcase_conj}
	(a_{-t_n}g)^{-1}\gamma_n(a_{-t_n}g) \in \Gamma.
	\end{equation}
	
	Note that for all $ s \in \R^{d-1} $ we have
	\[ u(s)\gamma_n \to u(s)a^\prime=a^\prime u(kOs) \in NMAU \]
	where the set $ NMAU $ is open in $ G $. Let $ B $ be the open unit ball in $ \R^{d-1} $ (the specific norm is irrelevant); one can define smooth functions
	\[ \varphi_n: B \to \R^{d-1}, \quad v_n : B \to N \;\;\text{and}\;\;  \ell_n: B \to MA \]
	satisfying for all large $ n $ that
	\[ u(s) \gamma_n = v_n(s) \ell_n(s) u({\varphi_n(s)}). \]
	Note that the above sequences of functions converge uniformly on $ B $ as follows:
	\[ \varphi_n \to kO, \quad v_n \to e \;\;\text{and}\;\; \ell_n \to a'.  \]
	
	We may now compute
	\begin{align*}
	u({e^{t_n}s}) g\Gamma &=a_{t_n} u(s) a_{-t_n}  g \Gamma \overset{(\ref{tech_specialcase_conj})}{=} a_{t_n}(u(s)\gamma_n) a_{-t_n}g\Gamma=\\
	& = a_{t_n} v_n(s) \ell_n(s) u({\varphi_n(s)})a_{-t_n}g\Gamma = \\
	&= (a_{t_n} v_n(s) a_{-t_n}) \ell_n(s) u({e^{t_n}\varphi_n(s)})g\Gamma.
	\end{align*}
	
	For all large $ n $, the values of $ v_n $ are uniformly close to $ e $ while the conjugation $ a_{t_n} v_n(s) a_{-t_n} $ further contracts this term. In addition, $ \ell_n $ is uniformly close to $ a' $ implying
	\begin{equation}\label{tech_specialcase_distance}
	d_{\GmodGamma} \left(u({e^{t_n}s})x,a^\prime u({e^{t_n}\varphi_n(s)}) x \right) < \min \{\delta,\varepsilon\}
	\end{equation}
	for all large $ n $, where $ \delta $ is the constant in the construction of $ f_\bot $ in \cref{lma_witness}. We consequently deduce
	\begin{equation}\label{tech_specialcase_f2fbot}
	f(u({e^{t_n}s})x)\leq f_\bot (u({e^{t_n}\varphi_n(s)})) 
	\end{equation}
	for all $ s \in B $.
	
	Note that while contracting the errors in $ NMA $, conjugation by $ a_{t_n} $ expands the $ U $-direction. In particular, it is not necessarily true that $ e^{t_n} \varphi_n $ is close pointwise to $ e^{t_n}kO $ on $B$. Despite of this, integration over all of $ B $ under these two changes of variables is in fact comparable. More precisely, since $ \varphi_n \to kO $ uniformly on $ B $ and $k<1$ we know 
	\begin{equation}\label{tech_specialcase_jacobian}
	\min_{s\in B}|\textbf{Jac}_s(\varphi_n)| > \frac{1}{2}k
	\end{equation}
	and
	\begin{equation}\label{tech_specialcase_inclusion}
	\varphi_n(B)\subset B 
	\end{equation}
	for all large $ n $, where $ \textbf{Jac}_s (h) $ denotes the Jacobian of the function $ h $ at the point $ s $. This implies
	\begin{align*}
	\int_{e^{t_n}B}f(u(s) x)ds &=e^{t_n}\int_{B}f(u(e^{t_n} s) x)ds\\
	&\overset{(\ref{tech_specialcase_f2fbot})}{\leq} e^{t_n}\int_{B}f_\bot(u({e^{t_n}\varphi_n(s)}) x)ds\\
	&\overset{(\ref{tech_specialcase_jacobian})}{\leq} 
	2k^{-1}\int_{e^{t_n}\varphi_n(B)}f_\bot(u(s) x)ds \\
	&\overset{(\ref{tech_specialcase_inclusion})}{\leq} 2k^{-1}\int_{e^{t_n}B}f_\bot(u(s) x)ds.
	\end{align*}
	
	In conclusion, we have shown that for all large $ n $
	\[ 1 = \frac{\int_{\|s\|<e^{t_n}}f(u(s) x)ds}{\int_{\|s\|<e^{t_n}}f(u(s) x)ds} \leq 2k^{-1}\frac{\int_{\|s\|<e^{t_n}}f_\bot(u(s) x)ds}{\int_{\|s\|<e^{t_n}}f(u(s) x)ds} \]
	which is in contradiction to the ratio ergodic theorem and the choice of $ \eta $.
\end{proof}

\begin{rem}
	The seemingly innocuous condition \eqref{tech_specialcase_inclusion} with $B$ a ball centered at the origin  (which implies, since $x$ is a generic point, that the ratio ergodic theorem holds along $e^{t_n}B$ for the given sequence $t_n$) is what makes the special case of $ a' \in \asymneighborhood \cap (MA \smallsetminus M) $ so significantly simpler. For the general case, when $a'$ is the image under $\alpha$ of a point in $\asymneighborhood \smallsetminus MA $, one can obtain a similar matching of orbits, but for shifted balls $B_n$ (with the shift, that depends on $n$, bounded by a constant, and radius $\sim 1$); however, since we are classifying infinite measures very little can be said directly about how the points $u(e^{t_n}s).x$ for $s\in B_n$ reflect the behavior of the measure $\mu$, e.g.\ if there is even one $s \in B_n$ for which $u(e^{t_n}s).x$ is in an appropriate compact set of ``good'' points.
\end{rem}

\section{The Horosphere Reparameterization Map}\label{Section_reparam_map}

Let $ P=NMA $ be a minimal parabolic subgroup of $ G $. Let $ \pi : G \to P\backslash G $ be the projection onto $ P\backslash G \cong M \backslash K \cong S^{d-1} $. Identifying $ G $ with $ \FHd $, the frame bundle of hyperbolic $ d $-space, this quotient map corresponds to the projection of frames onto the limit points in $ \partial \Hd $ they tend to under the geodesic flow $ a_t $ as $ t \to +\infty $. 

Let $ \omega \in G $ be any representative of the non-trivial coset in the Weyl group of $ G $, i.e $ e\neq [\omega] \in W={C_G(A)} \backslash N_G(A) \cong \Z_2 $ and $ P\omega \neq P $. Using the fact that $ \omega N\omega=U $ one can reformulate the Bruhat decomposition of $ G $ as
\[ G = G\omega = (P\omega P \cupdot P)\omega = PU \cupdot P\omega  \]
where $ PU $ is a Zariski open set in $ G $. Note that $ U\cap P=\{e\} $ implying that the restricted function
\[ \pi|_U : U \to (P\backslash G) \smallsetminus \{P\omega\}  \]
is a bijection. Any element $ \gamma \in G $ acts by right multiplication on $ P\backslash G \cong S^{d-1} $ as a M{\"o}bius transformation. This action commutes with $ \pi $ therefore any horosphere $ U\gamma $ projects injectively via $ \pi $ onto $ (\PbackslashG)  \smallsetminus \{P\omega\gamma\} $. Let 
\[ v_{\gamma}=(\pi|_U)^{-1}(P\omega \gamma^{-1})  \]
be the point in $ U $ satisfying $ \pi(v_{\gamma}\gamma)=P\omega $, whenever such point exists.

Denote by $ \iota:N\times MA \times U \hookrightarrow G $ the multiplication map $ (n,ma,u)\mapsto nmau $. This map is an injective diffeomorphism onto its image $ PU $. 

Let $ \Phi=(\Phi_N,\Phi_{MA},\Phi_U) $ be the following smooth map
\[ \Phi:\;G\times U \smallsetminus \{(g,v_{g})\}_{g\in G} \to N \times MA \times U \]
defined by
\[ \Phi(\gamma,u)=\iota^{-1}(u\gamma). \]

Note that for any point $ g\Gamma \in \GmodGamma $, any element $ \gamma^\prime \in \Gamma $ and corresponding
\[ \gamma=g\gamma^\prime g^{-1} \in \text{Stab}_G(g\Gamma), \]
the following identity holds for all $ u \in U $ in the domain of definition of $ \Phi $:
\[ ug\Gamma=u\gamma g \Gamma = \Phi_{N}(\gamma,u) \Phi_{MA}(\gamma,u) \Phi_{U}(\gamma,u) g \Gamma. \]
Therefore $ \Phi $ may be viewed as a reparameterization map for the horospherical orbit $ Ug\Gamma $ in $ \GmodGamma $.

\vspace{1.5mm}
Corresponding to any hyperbolic (loxodromic) $ \gamma \in G $ there exist two distinct fixed points for the action of $ \gamma $ on $ P\backslash G $ by $ Pg\mapsto Pg\gamma $, one attracting fixed point and one repelling fixed point. 
Let $ Pk_\gamma^- $ be the attracting fixed point of $ \gamma $, then the map $ \gamma \mapsto Pk_\gamma^- $ defined on the open subset of hyperbolic elements in $ G $ is continuous. 
For any hyperbolic $ \gamma $ with $ Pk_\gamma^-\neq P\omega $ denote
\[ u_\gamma = (\pi|_U)^{-1}(Pk_\gamma^-). \]
Since $ \pi(u_\gamma\gamma)=\pi(u_\gamma) $ one has $ \Phi_U(\gamma,u_\gamma)=u_\gamma $ and
\begin{equation}\label{tech_Phi2alpha}
u_\gamma \gamma=\Phi_N(\gamma,u_\gamma)\Phi_{MA}(\gamma,u_\gamma)u_\gamma.
\end{equation}
The fact that we have chosen the point $ u_\gamma $ with respect to the attracting fixed point of $ \gamma $, together with \eqref{tech_Phi2alpha} ensures that $ \Phi_{MA}(\gamma,u_\gamma) $ satisfies both of the conditions of $ \alpha(\gamma) $ in \cref{dfn_alpha}, hence in these notations
\[ \alpha(\gamma)=\Phi_{MA}(\gamma,u_\gamma). \]

Note that whenever $ Pk_\gamma^-=P\omega $ then $ Pk_{\gamma^{-1}}^-\neq P\omega $. This ensures that given any hyperbolic element $ \gamma $, at least on of $ \alpha(\gamma) $ and $ \alpha(\gamma^{-1}) $ is non-trivial.

\begin{rem}
	It is worth noting that $ \gamma $ and $ \alpha(\gamma) $ are conjugate, as (\ref{tech_Phi2alpha}) implies 
\begin{equation}\label{tech_alpha_conjugacyrelation}
\exists n\in N \quad \alpha(\gamma) = (nu_\gamma) \gamma (nu_\gamma)^{-1}.
\end{equation}
Therefore $ \alpha(\gamma) $ may be thought of as a representative in $ MA $ of the action of $ \gamma $ along its axis in $ \FHd $.
\end{rem}

Let $ R: G \acts \PbackslashG $ be the smooth (contravariant) action on $ \PbackslashG $ by right multiplication $ R_\gamma: Pg \mapsto Pg\gamma $.  Notice that for any $ \gamma \in G $ the maps $ \Phi_U(\gamma,\cdot) $ and $ R_{\gamma} $ are conjugate since
\begin{equation}\label{tech_lma_conjPhi}
\Phi_U(\gamma,\cdot)= (\pi|_U)^{-1}\circ R_{\gamma} \circ (\pi|_U).
\end{equation} 
The following technical lemma is a consequence of \eqref{tech_lma_conjPhi} together with basic properties of the action of hyperbolic elements on $ \PbackslashG \cong \partial \Hd $.

\medskip
Fix some bi-invariant metrics on $ U, N $ and $ MA $ and denote $ B^U_r, B^N_r, B^{MA}_r $ the open balls of radius $ r $ around the identity in these groups respectively. As before, let $ \textbf{Jac}_u(F) $  denote the Jacobian of the function $ F $ at the point $ u $.

\begin{lemma}\label{lma-Phi}
	Let $ \gamma_0 $ be any hyperbolic element of $ G $. Let $ Pk_{\gamma_0}^- $ and $ Pk_{\gamma_0}^+ $ be the attracting and repelling fixed points for $ \gamma_0 $ in $ \PbackslashG $ respectively. If $ Pk_{\gamma_0}^-\neq P\omega $ and $ Pk_{\gamma_0}^+ \neq P $ then for any $ \varepsilon $ there exist:
	\begin{itemize}[leftmargin=*]
		\item an open neighborhood $ V_0 \subset G $ of $ \gamma_0 $
		\item constants $ k \in \N $, $ 0 < c $, $ 0<R $
		\item balls $ D $ and $ I $ around $ e $ in $ U $
	\end{itemize}
	such that for any $ \gamma \in V_0 $:
	\begin{enumerate}[leftmargin=*]
		\item the map $ \Phi(\gamma^k,\cdot) $ is defined and smooth on $ D $
		\item the map $ \Phi(\gamma,\cdot) $ is defined and smooth on $ E=I \cdot \Phi_U(\gamma^k,D) $
		\item $ B^U_{cr} \subseteq \Phi_U(\gamma^k,B^U_{r}) \cdot (\Phi_U(\gamma^k,e))^{-1} \subseteq B^U_{r} $ for all $ r > 0 $ satisfying $ B_r^U \subseteq D $
		\item $ \forall u \in E\; \forall v \in D \quad c<|\mathbf{Jac}_{u}(\Phi_U(\gamma,\cdot))| < 1 $ and $ c < |\mathbf{Jac}_{v}(\Phi_U(\gamma^k,\cdot))| < 1 $
		\item $ \Phi(\gamma^k,D) \subseteq B_R^N\times B_R^{MA}\times B_R^U $
		\item $ \Phi(\gamma,E) \subseteq B_R^N\times B_R^{MA}\times B_R^U $
	\end{enumerate}
	and moreover:
	\begin{enumerate}[leftmargin=*]
		\setcounter{enumi}{6}
		\item $ \Phi_{MA}(\gamma,E)\subseteq B^{MA}_\varepsilon(\alpha(\gamma_0)) $, where $ B^{MA}_\varepsilon(\alpha(\gamma_0)) $ denotes the ball of radius $ \varepsilon $ around $ \alpha(\gamma_0) $ in $ MA $
		\item $ \forall v \in D \;\; \Phi_U(\gamma,I\cdot \Phi_U(\gamma^k,v)) \subseteq I \cdot \Phi_U(\gamma^k,v) $
	\end{enumerate}
\end{lemma}

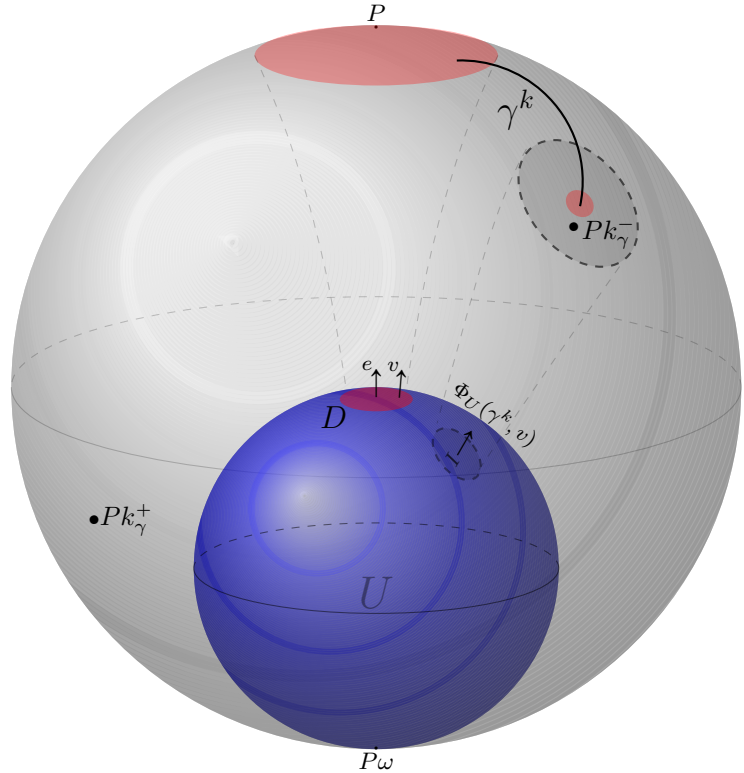
\begin{figure}[b]\label{Figure_lma_Phi}
	\begin{tikzpicture}[scale=0.8]
	\def\R{6};
	\def\r{3};
	
	\fill[ball color=white, fill opacity=0.2] (0,0) circle (\R);
	\draw[opacity=0.2]  (-\R,0) arc (180:360:6cm and 1.5cm);
	\draw[dashed, opacity=0.2]  (\R,0) arc (0:180:6cm and 1.5cm);
	
	\fill[ball color=blue, fill opacity=0.2] (0,\r-\R) circle (\r);
	\draw[opacity=0.5]  (-\r,\r-\R) arc (180:360:3cm and 0.25*3cm);
	\draw[dashed, opacity=0.5]  (\r,\r-\R) arc (0:180:3cm and 0.25*3cm);
	
	\draw[thick, dashed, fill=black, rotate around={40:(0,0)}, fill opacity=0.1, draw opacity=0.6] (4.5,0.2) ellipse (0.79cm and 1.2cm);
	
	\draw (3.25,2.65) node {\tiny \textbullet};
	\draw (3.2,2.55) node[anchor=west] {\footnotesize $ Pk_{\gamma}^- $};
	\draw (-4.2,-2.2) node {\tiny \textbullet \footnotesize $ Pk_{\gamma}^+ $};
	\draw (0,\R-0.04) node {$ \cdot $};
	\draw (0,\R+0.2) node {\scriptsize $ P $};
	\draw (0,-\R) node {$ \cdot $};
	\draw (0,-\R-0.2) node {\scriptsize $ P\omega $};
	
	\draw[thick, dashed, fill=black, rotate around={40:(0,0)}, fill opacity=0.1, draw opacity=0.4] (0.3,-1.7) ellipse (0.3cm and 0.5cm);
	\draw[color=black, opacity=0.4] (0,\r-\R-0.4) node {\Large $ U $};
	
	\fill[color=red, fill opacity=0.3]  (0,5.5) ellipse (2 and 0.5);
	
	\draw[fill=red, fill opacity=0.3, draw opacity=0, rotate around={48:(0,0)}]  (4.5,-0.45) ellipse (0.2 and 0.25);
	
	\fill[color=red, opacity=0.4]  (0,-0.2) node (v1) {} ellipse (0.6 and 0.2);
	
	\draw[thick] (3.35,3) arc (-12:92:2);
	\draw (2.3,4.6) node {\large $\gamma^k$};
	
	\draw (-0.7,-0.5) node {$ D $};
	
	\draw[dashed,opacity=0.2]  plot[smooth, tension=.7] coordinates {(-2,5.5) (-1.4,4) (-0.85,2) (-0.5,2*\r-\R)};
	\draw[dashed,opacity=0.2]  plot[smooth, tension=.7] coordinates {(2,5.5) (1.4,4) (0.85,2) (0.5,2*\r-\R)};
	
	\draw[dashed,opacity=0.2]  plot[smooth, tension=.7] coordinates {(2.47,3.95) (1.9,2.6) (1,-0.6)};
	\draw[dashed,opacity=0.2]  plot[smooth, tension=.7] coordinates {(4,2) (2.75,0.2) (1.6,-1.5)};
	
	\draw (0,0.07) node {$ \uparrow $};
	\draw (-0.15,0.1) node[anchor=south] {\tiny $ e $};
	
	\node[rotate=-5] at (0.4,0.05) {$ \uparrow $} ;
	\draw (0.28,0.08) node[anchor=south] {\tiny $ v $};
	
	\node[rotate=-30] at (1.45,-0.92) {$ \uparrow $} ;
	\node[rotate=-35] at (1.95,-0.45) {\tiny $ \Phi_U(\gamma^k,v) $};
	
	\node[rotate=60, opacity=0.7] at (1.25,-1.20) {\tiny $ I $};
	
	\end{tikzpicture}
	\caption{\cref{lma-Phi} in the hyperbolic ball model. Here we identify $ \PbackslashG  $ with $ \partial\mathbb{H}^3=S^{2} $, and the subgroup $ U $ with the image of $ Ue $ under the projection $ G \to T^1\mathbb{H}^3 $ to the unit tangent bundle of $ \mathbb{H}^3 $.}
\end{figure}

\begin{proof}
	Use the map $ \pi|_U $ to induce a metric $ d_\pi $ onto $ (\PbackslashG) \smallsetminus \{P\omega \} $ satisfying for any $ Pg \neq P\omega $ and $ 0<r $, 
	\[ B^{d_\pi}_r(Pg)=\pi \left(B^U_r\left((\pi|_U)^{-1}(Pg)\right)\right). \]
	
	Since the map $ R_{\gamma_0} $ contracts points near $ k_{\gamma_0}^- $, there exist $ 0< r_1 < r_2 $ and a neighborhood $ V_1 \subset G $ of $ \gamma_0 $, satisfying for every $ \gamma \in V_1 $
	\[ R_\gamma(B^{d_\pi}_{r_2}(Pk_{\gamma_0}^-)) \subseteq  B^{d_\pi}_{r_1}(Pk_{\gamma_0}^-). \]
	Moreover, we can choose $ r_2 $ small enough so that
	\[ P\omega\gamma^{-1} \notin B^{d_\pi}_{r_2}(Pk_{\gamma_0}^-). \]
	Since $ Pk_{\gamma_0}^+\neq P $,
	\[ P\gamma_0^\ell \to Pk_{\gamma_0}^- \text{ and } P\omega\gamma_0^{-\ell} \to Pk_{\gamma_0}^+ \quad \text{as } \ell \to \infty. \]
	Hence there exist $ 0<\rho, k $ and a neighborhood $ V_0 \subset V_1 $ of $ \gamma_0 $ satisfying for every $ \gamma \in V_0 $ and every $ Pg \in B^{d_\pi}_\rho (P) $
	\[ B^{d_\pi}_{r_1}(Pk_{\gamma_0}^-)) \subseteq B^{d_\pi}_{\nicefrac{(r_2+r_1)}{2}}(R_{\gamma^k}(Pg)) \subseteq B^{d_\pi}_{r_2}(Pk_{\gamma_0}^-)) \]
	and
	\[ P\omega\gamma^{-k} \notin \overline{B^{d_\pi}_\rho (P)}. \]
	This implies in particular
	\begin{equation}\label{tech_inclusion_in_PbackslashG}
	R_\gamma \left(B^{d_\pi}_{\nicefrac{(r_2+r_1)}{2}}(R_{\gamma^k}(Pg))\right) \subseteq B^{d_\pi}_{\nicefrac{(r_2+r_1)}{2}}(R_{\gamma^k}(Pg)).
	\end{equation}
	
	Let $ D=B^U_{\rho} $ and $ I=B^U_{\nicefrac{(r_2+r_1)}{2}} $. For all $ \gamma \in V_0 $ the maps $ \Phi(\gamma^k,\cdot) $ and $ \Phi(\gamma, \cdot) $ are defined on $ \overline{D} $ and 
	\[ \overline{I\cdot (\pi|_U)^{-1}\left(R_{\gamma^k}\left(B^{d_\pi}_{\rho}(P)\right)\right)} \] respectively. Item (8) is directly implied from (\ref{tech_inclusion_in_PbackslashG}) using the conjugation relation \eqref{tech_lma_conjPhi}. Continuity of the map $ \Phi $ implies item (7), after further restriction of $ V_0 $ if necessary. 
	After possibly increasing $ k $, items (3) and (4) follow from the fact that the maps $ \Phi_U(\gamma^k,\cdot) $ and $ \Phi_U(\gamma,\cdot) $ have non-vanishing Jacobians on the compact sets 
	\[ \overline{D} \ \text{ and }\ \overline{I \cdot \Phi_U(\gamma^k,D)}, \] by (\ref{tech_lma_conjPhi}), and are contracting.
\end{proof}

\section{Proof of \cref{main_theorem}}\label{Section_proof_mainthm}

Given an element $ g \in G $ or a set $ F \subset G $ we denote $ g^{a_t}:=a_t g a_{-t} $ and similarly $ F^{a_t}:= a_t F a_{-t}$, for any $ t \in \R $. Recall that $ B^U_r $ denotes the ball of radius $ r $ around $ e $ in $ U $ w.r.t. some fixed bi-invariant norm. Similarly for $ B^N_r $ and $ B^{MA}_r $.

Given functions $ h_1,h_2 \in L^1(\GmodGamma) $, a point $ x \in \GmodGamma $ and a bounded measurable set $ E \subset U $ denote
\[ \mathcal{R}(h_1,h_2,x,E) = \frac{\int_E h_1(u x)du}{\int_E h_2(u x)du}. \]
Since $a_t$ normalizes $U$, and expands $U$ for $t>0$, the ratio ergodic theorem may be stated as follows: fixing  $ B=B_r^U $,
\[ \lim_{t \to \infty}\mathcal{R}(h_1,h_2,x,B^{a_t}) = \frac{\int h_1 d\mu}{\int h_2 d\mu} \qquad \text{for $ \mu $-a.e.~$ x \in \GmodGamma $},\]
whenever $ \int h_2 d\mu \neq 0 $.

\medskip
We now begin the proof of \cref{main_theorem}.
Let $ \gamma_0 \in \asymneighborhood $ be a hyperbolic element for which $ \alpha(\gamma_0) \in MA $ is well defined, i.e. $ Pk_{\gamma_0}^-\neq P\omega $ where $ Pk_{\gamma_0}^- $ is the attracting fixed point of $ \gamma_0 $ in $ \PbackslashG  $. 

Without loss of generality we may assume the repelling fixed point admits $ Pk_{\gamma_0}^+ \neq P $. Indeed, in the case where $ Pk_{\gamma_0}^+ = P $ then $ \gamma_0 \in P $, implying both $ \alpha(\gamma_0) $ and $ \alpha(\gamma_0^{-1}) $ are non-trivial and are inverses of one another, in which case we may replace $ \gamma_0 $ by $ \gamma_0^{-1} $.

We need to show $ \alpha (\gamma_0) \in H_\mu $. Assume otherwise in contradiction. Then there exists an $ \varepsilon > 0 $ for which $ B^{MA}_\varepsilon(\alpha(\gamma_0)) $, the $ \varepsilon $-neighborhood of $ \alpha(\gamma_0) $ in $ MA $, is disjoint from $ H_\mu $. Let $ V_0,D,I,k,c,R $ be the sets and constants chosen in \cref{lma-Phi} for $ \varepsilon $ and $ \gamma_0 $. 

Fix some $ 0<\eta_1<\nicefrac{c}{2} $. By \cref{lma_witness} there exists a function witnessing the singularity $ \ell.\mu \bot \mu $ for all $ \ell \in B^{MA}_\varepsilon\left( \alpha (\gamma_0) \right) $, more precisely there exists $ f \in C_c(\GmodGamma) $ and a constant $ \delta > 0 $ for which the function
\[ f_\bot(x)=\max \{ f(v \ell x) : v \in B^N_\delta(e),\; \ell \in B^{MA}_\varepsilon\left( \alpha (\gamma_0) \right) \} \]
satisfies
\[ \frac{\int  f_\bot d\mu}{\int f d\mu} <\eta_1. \]

Fix some open precompact set $ \Omega \subset \GmodGamma $ with $ \mu(\Omega) > 0 $ and let $ \tOmega $ be any compact set containing a neighborhood of 
\[ \bigcup_{\ell \in B_R^{MA}} \ell^{-1} \Omega. \]

Let $ \eta_2 > 0 $ be sufficiently small, to be made explicit later. By the ratio ergodic theorem we may choose a constant $ T_1 > 0 $ large enough so that
\begin{equation}\label{def_tOmega_1}
\tOmega_1 = \left\{ x \in \tOmega \;\middle|\; 
\mathcal{R}(f_\bot,f,x,I^{a_t})<\eta_1
\;\;  \forall t \geq T_1 \right\} 
\end{equation}
satisfies 
\[ \frac{\mu(\tOmega \smallsetminus \tOmega_1)}{\mu(\tOmega)}<\frac{1}{4}\eta_2. \]
We additionally require that $ T_1 $ be large enough so that for any $ n \in B_R^N $ and $ \ell \in B_R^{MA} $
\begin{equation}\label{tech_Omeg2tOmega}
(n^{a_{T_1}} \cdot \ell)^{-1} \Omega \subseteq \tOmega
\end{equation}
and
\begin{equation}\label{tech_N_error}
n^{a_{T_1}} \in B^N_{\delta} 
\end{equation}
with $ R $ as in \cref{lma-Phi}.

Since $ \asymneighborhood \cap V_0 \neq \emptyset $ there exists some $ T_2 > T_1 $ for which the set
\[ \tOmega_2 = \left\{ x=g\Gamma \in \tOmega_1 \; \middle|\; \bigcup_{t \in [T_1,T_2]} a_{-t}g\Gamma g^{-1}a_t \cap V_0 \neq \emptyset \right\} \]
satisfies
\[ \frac{\mu(\tOmega \smallsetminus \tOmega_2)}{\mu (\tOmega)}< \min \left\{\frac{\eta_2}{2},\frac{\mu(\Omega)}{4\mu(\tOmega)}\right\}. \]
This condition together with the ratio ergodic theorem ensures the existence of a point $ x_0 \in \GmodGamma $ and a constant $ L > e^{T_2} $ satisfying
\begin{equation}\label{tech_ineq_x0_1}
\mathcal{R}(\chi_{\tOmega_2},\chi_{\tOmega},x_0,B^U_L)>1-\eta_2 
\end{equation}
\begin{equation}\label{tech_ineq_x0_2}
\mathcal{R}(\chi_{\Omega \cap \tOmega_2},\chi_{\tOmega},x_0,B^U_L) > \frac{\mu(\Omega)}{2\mu(\tOmega)} 
\end{equation}

A key component of Hochman's proof of the ratio ergodic theorem is a multi-dimensional variant of the Chacon-Ornstein Lemma \cite[Thm.~1.2]{Hochman2010} stating that for any fixed $ r > 0 $:
\begin{equation}
\frac{\int_{\partial_{(r)}B^U_R} h_1 (ux)du}{\int_{B^U_R} h_2 (ux)du} \underset{R\to \infty}{\longrightarrow} 0 
\end{equation}
where $ \partial_{(r)} B^U_R=B^U_{R}\smallsetminus B^U_{R-r} $ and $ h_1,h_2 \in L^\infty (\mu) \cap L^1(\mu) $ with $ \int h_2 d \mu \neq 0 $. We may thus choose $ L $ large enough so as to additionally satisfy
\begin{equation}\label{tech_ineq_x0_Chacon-Ornstein}
\int_{\partial_{(R^\prime e^{T_2})}B^U_L} \chi_{\Omega \cap \tOmega_2}(u x_0)du < \frac{1}{2} \int_{B^U_L} \chi_{\Omega \cap \tOmega_2}(u x_0)du
\end{equation}
where $ R^\prime=R+\text{radius}(D) $.

In what follows, we shall focus our attention on the orbit $ \{u x_0\}_{u\in B^U_L} $ of $ x_0 $. Let $ m_U $ denote Haar measure on $ U $. Given a set $ E $ in $ G $ we denote 
\[ O_E=\{u \in B^U_L | ux_0 \in E \}. \]
Sets of particular interest are $ O_{\Omega} $, $ O_{\tOmega} $ and $ O_{\tOmega_2} $.

Assumption (\ref{tech_ineq_x0_1}) implies that out of the time in which $ x_0 $'s $ U $-orbit spends inside $ \tOmega $, close to a full proportion of that time is spent inside the set $ \tOmega_2 $, i.e. $ \nicefrac{m_U(O_{\tOmega} \smallsetminus O_{\tOmega_2})}{ m_U(O_{\tOmega})} $ is small (less than $ \eta_2 $). The heart of the proof is an argument employing the map $ \Phi_U $ to ``push'' an impossible proportion of $ O_\Omega $ into $ O_{\tOmega} \smallsetminus O_{\tOmega_2} $. 

We will show there exists a constant $ C_0 > 0 $ independent of $ \eta_2 $ for which
\begin{equation}\label{key_prop}
m_U (O_{\tOmega} \smallsetminus O_{\tOmega_2})
\geq C_0 \cdot m_U(O_{\Omega} \cap O_{\tOmega_2})
\end{equation}
implying
\[ \eta_2 \underset{(\ref{tech_ineq_x0_1})}{>} \frac{m_U (O_{\tOmega} \smallsetminus O_{\tOmega_2})}{m_U (O_{\tOmega})} \geq C_0 \frac{m_U (O_{\Omega} \cap O_{\tOmega_2})}{m_U (O_{\tOmega})} \underset{(\ref{tech_ineq_x0_2})}{>} C_0 \frac{\mu(\Omega)}{2\mu(\tOmega)}. \]
Taking $ \eta_2 $ small enough yields the desired contradiction.

In order to prove (\ref{key_prop}) we construct a family of open sets contained in $ O_{\tOmega} \smallsetminus O_{\tOmega_2} $ and use a covering argument to bound the total mass of these sets.

Fix $ g_0 \in G $ for which $ x_0=g_0 \Gamma $. By definition, for every $ u_0 \in O_{\tOmega_2} $ there exists $ t_{u_0} \in [T_1,T_2] $ and $ \gamma_{u_0} \in G $ with
\begin{equation}\label{tech_gamma_conjugation}
\gamma_{u_0} \in \left(a_{-t_{u_0}} (u_0 g_0)\Gamma (u_0 g_0)^{-1}a_{t_{u_0}}\right) \cap V_0
\end{equation}
with $ V_0 \subset G $ as in \cref{lma-Phi}. 
To the element $ u_0 \in O_{\tOmega_2} $ we associate the following auxiliary function $ \Psi_{u_0}: D \to U $ defined by
\[ \Psi_{u_0}(v)=\left(\Phi_U\left(\gamma_{u_0}^k,v \right)\right)^{a_{t_{u_0}}}u_0. \]

\begin{lemma}\label{lma_case1}
	For all $ u_0 \in O_{\tOmega_2} $
	\begin{equation}\label{tech_J_disjoint}
	\Psi_{u_0}(D) \cap O_{\tOmega_2} = \emptyset.
	\end{equation}
\end{lemma}

The proof of this lemma follows along the lines of \cref{thm_special_case} given in \S\ref{Section_Special_Case}. Recall that points in $ O_{\tOmega_2} $ have a ``good'' ratio ergodic average, in the sense of (\ref{def_tOmega_1}).  In \cref{thm_special_case}, we have shown that given a $ \mu $-generic point whose geodesic past at time $ t $ is close to a closed geodesic of some approximate type, we may reparameterize its horospherical orbit to deduce it has a ``bad'' ratio ergodic average over a ball in $ U $ of size $ e^t $. In the case where $ t $ may be taken to be arbitrarily large we derive a direct contradiction to the ratio ergodic theorem. 

The underlying reason that was implicitly behind the argument establishing that the ratio ergodic average on a ball of size $ e^t $ is ``bad'' in this case is the fact that $ \Phi_U(\gamma_{u_0},\cdot) $ maps a ball around the identity into itself whenever $ \gamma_{u_0} $ is sufficiently close to a contracting normalizing element of $ U $. For a general $ \gamma_{u_0} $ this is not necessarily true. 

By the way we have set things up, balls centered at points in $ \Psi_{u_0}(D) $ are mapped into a subset of themselves by $ \Phi_U(\gamma_{u_0},\cdot) $, in which case ``bad'' ratio ergodic averages are implied, that is to say these points are not contained in $ O_{\tOmega_2} $, as claimed in (\ref{tech_J_disjoint}).

\begin{figure}[b]
	\begin{tikzpicture}[scale=0.8]
	
	\node at (-5.35,-5.15) {\footnotesize \textbullet $ u_0 $};
	
	\draw[fill=black, fill opacity=0.04]  (-5.5,-5) ellipse (1.5 and 1.5);
	\node at (-7.1,-6) {$Du_0$};
	
	\node at (-6.2,-4.5) {\footnotesize \textbullet};
	\draw[opacity=0.4] (-5.55,-5.15) --  (-6.2,-4.5);
	\node at (-5.75,-4.65) {\tiny $ v $};
	
	\node at (-0.4,-1.55) {\footnotesize \textbullet $w$};
	
	\draw[opacity=0.5] (-0.62,-1.46) node {} ellipse (2.5 and 2.5);
	\node at (-2.75,-3.5) {$I^{a_t} w$};
	
	\draw[fill=black, fill opacity=0.04]   (-0.4,-1.6) ellipse (0.5 and 0.4);
	\node at (-0.25,-2.25) {\footnotesize $\Psi_{u_0}(D)$};
	
	\draw[dashed] (-0.6,-1.5) arc (40:200:3.08);
	\node at (-5,-0.55) {$ \Psi_{u_0} $};
	
	\end{tikzpicture}
	\caption{Elements in \cref{lma_case1} associated to a point $ u_0 \in O_{\tOmega_2} $.}
\end{figure}
\begin{proof}[Proof of \cref{lma_case1}]
	Fix $ u_0 \in O_{\tOmega_2} $ and set $ \gamma=\gamma_{u_0} $ and $ t = t_{u_0} $. By item (1) of \cref{lma-Phi} the map $ \Phi(\gamma^k,\cdot) $ is well defined on $ D $. Fix $ v \in D $ and denote
	\begin{equation}\label{tech_u1_def}
	w=\Psi_{u_0}(v)=(\Phi_U(\gamma^k,v))^{a_t}u_0.
	\end{equation}
	We would like to prove $ w \notin O_{\tOmega_2} $. To that end, it would suffice to show
	\[ \mathcal{R}(f_\bot,f,w x_0,I^{a_t}) > \eta_1. \]
	
	By (\ref{tech_gamma_conjugation}), the element
	\[ \gamma^\prime = (u_0 g_0)^{-1} a_t \gamma a_{-t} (u_0 g_0) \]
	is contained in $ \Gamma $. For any $ v $ in the domain of definition of $ \Phi_U(\gamma,\cdot) $ we have
	\begin{align}\label{eq_gamma2Phi}
	v^{a_t} u_0 g_0 \gamma^\prime &= a_t v a_{-t} u_0 g_0 \gamma^\prime= a_t (v \gamma) a_{-t} u_0 g_0= \nonumber\\
	&= a_t (\Phi_N(\gamma,v) \cdot \Phi_{MA}(\gamma,v) \cdot \Phi_U(\gamma,v)) a_{-t} u_0 g_0 = \\
	&= (\Phi_N(\gamma,v))^{a_t} \cdot \Phi_{MA}(\gamma,v) \cdot (\Phi_U(\gamma,v))^{a_t} \cdot u_0 g_0 \nonumber.
	\end{align}
	
	Denote
	\[ v'=\Phi_U(\gamma^k,v). \]
	By item (2) of \cref{lma-Phi} the map $ \Phi(\gamma,\cdot) $ is well defined on $ I \cdot v' $. Hence for any $ u \in I $ and $ w $ as in (\ref{tech_u1_def}), equation (\ref{eq_gamma2Phi}) applied to $ uv' $ implies
	\begin{align}\label{tech_reparam}
	u^{a_t} w x_0 =& (uv')^{a_t} u_0 x_0 \\
	=& (\Phi_N(\gamma,u v'))^{a_t} \cdot \Phi_{MA}(\gamma,u v') \cdot (\Phi_U(\gamma,u v'))^{a_t} \cdot u_0 x_0 \nonumber.
	\end{align}
	\cref{lma-Phi} ensures the different terms of (\ref{tech_reparam}) are well controlled. Item (6) states that $ \Phi_N(\gamma,u v') \in B^N_R $ hence by the choice of $ T_1 $ in (\ref{tech_N_error}) we are ensured that conjugation by $ a_t $ admits
	\[ \left(\Phi_N(\gamma,u v')\right)^{a_t} \in B_\delta^N. \]
	On the other hand item (7) of \cref{lma-Phi} gives that
	\[ \Phi_{MA}(\gamma,u v') \in B^{MA}_\varepsilon(\alpha(\gamma_0)). \]
	Recall the definition of the function $ f_\bot $:
	\[ f_\bot(x)=\max \{ f(v \ell x) : v \in B^N_\delta(e),\; \ell \in B^{MA}_\varepsilon\left( \alpha (\gamma_0) \right) \}. \]
	We thus see that
	\begin{align}\label{tech_f2fbot}
	f(u^{a_t} w x_0) \leq f_\bot\left(\left(\Phi_U(\gamma,u v')\right)^{a_t} \cdot u_0 x_0)\right).
	\end{align}
	
	In order to estimate the ratio ergodic average along $ I^{a_t}w x_0 $ we need some control over $ \Phi_U $. Item (8) of \cref{lma-Phi} applied to the point $ v' $ states that
	\begin{equation}\label{tech_PhiUcontract}
	\Phi_U(\gamma,I v') \subseteq I v'
	\end{equation}
	and item (4) states that
	\begin{equation}\label{tech_PhiUJacobian}
	c<|\textbf{Jac}_{u v'}(\Phi_U(\gamma,\cdot))|.
	\end{equation}
	In a manner analogous to the the proof of \cref{thm_special_case}, we arrive at the following calculation:
	\begin{align*}
	\int_{I}f(u^{a_t} w x_0)du & \overset{(\ref{tech_f2fbot})}{\leq}\int_{I}f_\bot\left(\left(\Phi_U(\gamma,u v')\right)^{a_t} \cdot u_0 x_0\right)du \leq \\
	& \overset{(\ref{tech_PhiUJacobian})}{\leq}  c^{-1} \int_{\Phi_U(\gamma,I v')}f_\bot(u^{a_t} u_0 x_0)du \leq \\
	& \overset{(\ref{tech_PhiUcontract})}{\leq}  c^{-1} \int_{Iv'}f_\bot(u^{a_t} u_0 x_0)du = \\
	& \overset{(\ref{tech_u1_def})}{=} c^{-1} \int_{I}f_\bot(u^{a_t} w x_0)du
	\end{align*}
	hence
	\begin{equation*}
	1=\mathcal{R}(f,f,w x_0,I^{a_t}) \leq c^{-1} \cdot \mathcal{R}(f_\bot,f,w x_0,I^{a_t}).
	\end{equation*}
	Since $ \eta_1<\nicefrac{c}{2} $ this implies $ w \notin O_{\tOmega_2} $ as required.
\end{proof}

\medskip
As a corollary of the lemma above we obtain the following:
\begin{lemma}
	For any $ u_0 \in O_{\tOmega_2} $
	\begin{equation}\label{tech_Jprime_contained}
	\Psi_{u_0}\left(D \cap (O_\Omega u_0^{-1})^{a_{-t}}\right) \subseteq O_{\tOmega} \smallsetminus O_{\tOmega_2}
	\end{equation}
	where $ t=t_{u_0} $ as above.
\end{lemma}

\begin{proof}
	Let $ v \in D \cap (O_\Omega u_0^{-1})^{a_{-t}} $. By the definition of $ O_\Omega $ we have
	\[ v^{a_t} u_0 x_0 \in O_\Omega x_0 \subset \Omega. \]
	On the other hand, as in (\ref{eq_gamma2Phi}), we know
	\begin{equation}
	v^{a_t} u_0 x_0 = (\Phi_N(\gamma^k,v))^{a_t} \cdot \Phi_{MA}(\gamma^k,v) \cdot (\Phi_U(\gamma^k,v))^{a_t} u_0 x_0
	\end{equation}
	therefore
	\[ \left(\Phi_U(\gamma^k,v)\right)^{a_t} u_0 x_0 \in  \left((\Phi_N(\gamma^k,v))^{a_t}\cdot \Phi_{MA}(\gamma^k,v)\right)^{-1} \Omega. \]
	
	Denote
	\[ n_v=\Phi_{N}(\gamma^k,v) \text{ and } \ell_v=\Phi_{MA}(\gamma^k,v). \]
	By item (5) in \cref{lma-Phi} we have $ n_v \in B_R^N, \ell_v \in B^{MA}_R $. The construction of $ \tOmega $ and the choice of $ T_1 $ in (\ref{tech_Omeg2tOmega}) ensure $ \left( n_v^{a_t}\cdot \ell_v \right)^{-1} \Omega \subseteq \tOmega $. Hence we may conclude
	\begin{equation*}
	\Psi_{u_0}(v) \in O_{\tOmega}.
	\end{equation*}
	Since $ \Psi_{u_0}\left(D \cap (O_\Omega u_0^{-1})^{a_{-t}}\right) $ is contained in $ \Psi_{u_0}(D) $ we deduce by \cref{lma_case1} it is disjoint from $ O_{\tOmega_2} $, implying the claim.
\end{proof}
We have thus shown that to every point in $ O_{\tOmega_2} $ corresponds an open set contained in $ O_{\tOmega} \smallsetminus O_{\tOmega_2} $. Our aim is to show these sets add up to a non-negligible $ m_U $ measure. This is the content of the following covering argument.

\medskip
Given any $ u_0 \in O_{\tOmega_2} $ we define
\[ J_{u_0}=\Psi_{u_0}\left(D^{a_{\ln (\nicefrac{c}{2})}} \cap (O_\Omega u_0^{-1})^{a_{-t_{u_0}}}\right). \]
Note that $ D^{a_{\ln (\nicefrac{c}{2})}} \subset D $, since $ c \in (0,1) $, and therefore 
\[ J_{u_0} \subseteq \Psi_{u_0}\left(D \cap (O_\Omega u_0^{-1})^{a_{-t_{u_0}}}\right) \subseteq O_{\tOmega} \smallsetminus O_{\tOmega_2}. \]

\begin{lemma}\label{lma_cover}
	There exists a constant $ C_0 > 0 $, depending only on the group $ U $, the set $ D $ and the constants $ c $ and $ R $ of \cref{lma-Phi}, satisfying
	\[ m_U\left( \bigcup_{u_0 \in O_{\tOmega_2}} J_{u_0} \cap B_{L}^U \right) > C_0 \cdot m_U(O_{\Omega} \cap O_{\tOmega_2}). \]
\end{lemma}
\noindent
Note that the set on the left hand side is open in $ U $ and hence measurable. Recall the constant $ L $ was chosen to satisfy \eqref{tech_ineq_x0_Chacon-Ornstein}.
\begin{proof}
	The set $ D $ constructed in \cref{lma-Phi} is a ball in $ U $ around the identity. For convenience we will use the metric notation $ D=B_\rho^U $ throughout the proof. Recall 
	\[ \left(B^U_r\right)^{a_s} = a_s B^U_r a_{-s} = B^U_{r e^s} \]
	for any $ r > 0 $ and $ s \in \R $. 
	
	Set $ R'=R+\rho $. Consider the following open cover of $ (O_{\Omega} \cap O_{\tOmega_2}) \cap B^U_{L-R^\prime e^{T_2}} $:
	\[ \mathcal{F} = \left\{ B^U_{\frac{1}{2}c \rho e^{t_{u_0}}} u_0 \;:\; u_0 \in (O_{\Omega} \cap O_{\tOmega_2}) \cap B^U_{L-R^\prime e^{T_2}} \right\}.  \]
	By the Besicovitch covering theorem \cite[Theorem 2.7]{mattila1999geometry} there exists a constant $ C_{\text{Bes}} \geq 1 $, independent of $ \mathcal{F} $ and of $ (O_{\Omega} \cap O_{\tOmega_2}) \cap B^U_{L-R^\prime e^{T_2}} $, and a countable sub-cover $ \mathcal{F}^\prime \subset \mathcal{F} $ of balls centered at $ F=\{u_1,u_2,...\} $ for which 
	\begin{equation}\label{tech_Besicovitch_const_P}
	\sum_{B' \in \mathcal{F}'} \chi_{B'} \leq C_{\text{Bes}}.
	\end{equation} 
	
	We will focus our attention to the associated collection of sets $ \{J_{u_i}\}_{u_i \in F} $. 
	For any $ u_i \in F $ choose $ t_i \in [T_1,T_2] $ and $ \gamma_i \in G $ satisfying
	\begin{equation*}
	\gamma_i \in (a_{-t_i} (u_0 g)\Gamma (u_0 g)^{-1}a_{t_i}) \cap V_0.
	\end{equation*}
	Item (3) in \cref{lma-Phi} gives that
	\begin{equation}\label{tech_item3_lma_Phi}
	B^U_{cr}\cdot \Phi_U(\gamma_i^k,e) \subseteq \Phi_U(\gamma_i^k,B^U_{r}) \subseteq B^U_{r} \cdot \Phi_U(\gamma_i^k,e)
	\end{equation}
	for all $ 0 < r < \rho $.
	
	Recall the explicit definitions of $ \Psi_{u_i} $ and $ J_{u_i} $:
	\[ \Psi_{u_i}(\cdot)= \left(\Phi_U\left(\gamma_i^k,\cdot \right) \right)^{a_{t_i}}u_i \]
	\[ J_{u_i}= \left(\Phi_U\left(\gamma_i^k, B^U_{\frac{1}{2}c\rho} \cap (O_\Omega u_i^{-1})^{a_{-t_i}} \right) \right)^{a_{t_i}}u_i. \]
	If we denote
	\[ w_i=\Psi_{u_i}(e)=(\Phi_U({\gamma_{i}^k},e))^{a_{t_i}}u_i \]
	then the bounds given in (\ref{tech_item3_lma_Phi}) imply
	\begin{equation}\label{eq_Jprimes}
	J_{u_i} \subseteq B^U_{\frac{1}{2}c\rho e^{t_i}} w_i
	\end{equation}
	and
	\[ B^U_{c\rho e^{t_i}} w_i \subseteq \Psi_{u_i}(D) \subseteq B^U_{\rho e^{t_i}}w_i. \]
	Item (5) of \cref{lma-Phi} states, in particular, that $ \Phi_U(\gamma_{u_0}^k,e) \in B_R^U $ implying $ w_i \in B^U_{Re^{t_i}} u_i $. Therefore
	\begin{equation}\label{eq_Js}
	B^U_{c\rho e^{t_i}} w_i \subseteq \Psi_{u_i}(D) \subseteq B^U_{R^\prime e^{t_i}}u_i
	\end{equation}
	where $ R^\prime=R+\rho $ as defined earlier. 
	
	\begin{figure}[b]
		\begin{tikzpicture}[scale=0.455]
		\draw[fill=black, fill opacity=0.25,draw opacity=0.2]  plot[smooth cycle, tension=.8] coordinates {(-0.1,1.3) (-0.05,1) (0,0.5) (-0.5,0.6) (-1,0.35) (-0.5,1)};
		\draw[fill=black, fill opacity=0.25,draw opacity=0.2]  plot[smooth cycle, tension=.7] coordinates {(0.7,0.6) (1.15,0.7) (1.2,-0.1) (0.8,-0.1) (0.6,0.1)};
		\draw[fill=black, fill opacity=0.25,draw opacity=0.2]  plot[smooth cycle, tension=.7] coordinates {(1,-0.9) (0.5,-1.2) (0,-1) (1,-0.5)};
		
		\draw  (0,0) ellipse (1.5 and 1.5);
		
		\draw  (0,0) ellipse (3 and 3);
		
		\draw[rotate around={30:(0,0)},fill=black, fill opacity=0.05, draw opacity=0.2]  (0.4,0) ellipse (5 and 3.5);
		
		\draw[thick,rotate around={5:(-9,-6)}] (9,-6) node (v3) {} arc (0:60:18);
		\draw[thick,dashed,rotate around={-5:(-9,-6)}] (9,-6) node (v3) {} arc (0:10:18);
		\draw[thick,dashed,rotate around={65:(-9,-6)}] (9,-6) node (v3) {} arc (0:20:18);
		\draw[dashed,opacity=0.4,rotate around={5:(-9,-6)}]  (-9,-6) -- (9,-6);
		\draw[dashed,opacity=0.4,rotate around={65:(-9,-6)}]  (-9,-6) -- (9,-6);
		
		\draw  (-9,-6) ellipse (1.5 and 1.5);
		
		\node (v2) at (-9,-6) {\textbullet};
		\node at (-9.5,-6.2) {$u_i$};
		\node at (0,0) {\textbullet};
		\node at (-0.55,-0.3) {$w_i$};
		\node at (0.96,0.28) {\footnotesize $1$};
		\node at (0.5,1) {\footnotesize $2$};
		\node at (1.5,1.55) {\footnotesize $3$};
		\node at (3,2.55) {\footnotesize $4$};
		\node at (5,4) {\footnotesize $5$};
		\node at (-8.2,-5.45) {\footnotesize $6$};
		\node[draw,draw opacity=0.2,align=left] at (-9,6) {
			\textbf{\ \; Legend}\\\vspace{1pt} 
			$1$ --- $J_{u_i}$ \\\vspace{1pt} 
			$2$ --- $B^U_{\frac{1}{2}c\rho e^{t_i}} w_i$\\\vspace{1pt} 
			$3$ --- $B^U_{c\rho e^{t_i}} w_i$\\\vspace{1pt} 
			$4$ --- $\Psi_{u_i}(D)$\\\vspace{1pt} 
			$5$ --- $B^U_{R^\prime e^{t_i}}u_i$\\\vspace{1pt} 
			$6$ --- $B_i$};
		\end{tikzpicture}
		\caption{The different sets in \cref{lma_cover}} associated to a point $ u_i \in F $.
	\end{figure}
	
	We would like to bound $ \sum_{u_i \in F} \chi_{J_{u_i}} $. Set 
	\[ \beta=\frac{c\rho}{2R^\prime} < 1. \] 
	We claim that given any $ u_i,u_j \in F $ if $ J_{u_i} \cap J_{u_j} \neq \emptyset $ then
	\[ \beta \leq e^{t_j-t_i} \leq \beta^{-1} \]
	and
	\[ d_U(u_i,u_j) < 2\beta^{-1}R' e^{t_i} \]
	Indeed, if $ \beta \leq e^{t_j-t_i} \leq \beta^{-1} $ and $ \|u_i-u_j\|_U \geq 2\beta^{-1}R' e^{t_i} $ then
	\[ d_U(u_i,u_j) \geq \max\{2R' e^{t_i}, 2R' e^{t_j}\} \]
	and by (\ref{eq_Js}) we are ensured that $ \Psi_{u_i}(D) \cap \Psi_{u_j}(D) = \emptyset $ and consequently also $ J_{u_i} \cap J_{u_j} = \emptyset $. On the other hand, if $ e^{t_j-t_i} < \beta $ and $ v \in J_{u_i} \cap J_{u_j} $ then by (\ref{eq_Jprimes}) and (\ref{eq_Js}) we have
	\begin{align*}
	d_U(u_j,w_i) &\leq d_U(u_j,v)+d_U(v,w_i) < R' e^{t_j}+\frac{1}{2}c\rho e^{t_i} \\
	&< \beta R' e^{t_i}+\frac{1}{2}c\rho e^{t_i} = c\rho e^{t_i}
	\end{align*} 
	implying  $ u_j \in \Psi_{u_i}(D) $ and hence $ O_{\tOmega_2}\cap \Psi_{u_i}(D) \neq \emptyset $, in contradiction to \cref{lma_case1}, establishing the claim.
	
	\medskip
	To any $ u_j \in F $ let 
	\[ B_j = B^U_{\frac{1}{2}c \rho e^{t_j}} u_j \]
	be the corresponding ball in $ \mathcal{F}' $. We have thus shown that for any $ u_i $, if $ J_{u_j} $ intersects $ J_{u_i} $ then
	\[ \frac{1}{2}c \rho \beta e^{t_i} \leq \text{radius}(B_j) \leq \frac{1}{2}c \rho \beta^{-1} e^{t_i} \]
	and
	\[ B_j \cap B^U_{2\beta^{-1} R'e^{t_i}}u_i \neq \emptyset.\]
	Furthermore, Since $ c \rho < R' $, we have
	\[  B_j \subset B^U_{3\beta^{-1} R'e^{t_i}}u_i \]
	
	The number of such $ B_j $-s contained in $ B^U_{3\beta^{-1} R'e^{t_i}}u_i $ is bounded above by
	\begin{equation}\label{tech_coverconstC}
	C_{\text{Bes}} \frac{m_U(B^U_{3\beta^{-1} R'e^{t_i}})}{m_U(B^U_{\frac{1}{2}c \rho \beta e^{t_i}})} = C_{\text{Bes}} \left(\frac{6R'}{c \rho \beta^2}\right)^{d-1} =3^{d-1} C_{\text{Bes}} \left(\frac{2(R+\rho)}{c \rho}\right)^{3(d-1)}
	\end{equation}
	where $ d-1 = \dim U $ and $ C_{\text{Bes}} $ is the Besicovitch covering constant from (\ref{tech_Besicovitch_const_P}). Consequently if we denote the above constant in \eqref{tech_coverconstC} by $ C $ then
	\[ \sum_{u_i \in F} \chi_{J_{u_i}} \leq C. \]
	
	Note that since $ F \subseteq B^U_{L-R' e^{T_2}} $, by (\ref{eq_Js}) we know that all $ J_{u_i} $ are contained in $ B^U_L $.
	We may therefore conclude
	\[ m_U\left( \bigcup_{u \in O_{\tOmega_2}} J_u \cap B^U_L \right)  \geq m_U\left( \bigcup_{u_i \in F} J_{u_i} \right)
	\geq C^{-1}\sum_{u_i \in F} m_U(J_{u_i}). \]
	
	We shall now estimate the $ m_U $-measure of each $ J_{u_i} $. Let $ \varphi_i : U \to U $ be defined by 
	\[ \varphi_i(u)=u^{a_{t_i}}u_i  \]
	for all $ u \in U $. We may rewrite the definition of $ J_{u_i} $ as
	\begin{align*}
	J_{u_i} &= (\varphi_i \circ \Phi_U(\gamma_i^k,\cdot))\left(\left(B^U_{\frac{1}{2}c\rho e^{t_i}}u_i \cap O_\Omega\right)^{a_{-t_i}}u_i^{-1}\right) = \\
	& = (\varphi_i \circ \Phi_U(\gamma_i^k,\cdot) \circ \varphi_i^{-1})(B_i \cap O_\Omega)
	\end{align*}
	
	Note that the Jacobian $ \textbf{Jac}_u(\varphi_i) $ is constant $ \equiv e^{(d-1)t_i} $ over all of $ U $. By the uniform bound on the Jacobian of $ \Phi_U(\gamma_{u_0}^k,\cdot) $ on $ D $ given in item (4) of \cref{lma-Phi} we have
	\[ |\textbf{Jac}_{v}(\varphi_i \circ \Phi_U(\gamma_i^k,\cdot) \circ \varphi_i^{-1})| > c \]
	for all $ v \in D $, and consequently
	\[ m_U(J_{u_i}) > c \cdot m_U \left(B_i \cap O_\Omega\right). \]
	Hence
	\begin{align*}
	m_U\left( \bigcup_{u \in O_{\tOmega_2}} J_u \cap B^U_L \right) 
	& \geq C^{-1}\sum_{u_i \in F} m_U(J_{u_i}) \geq \\ 
	& \geq \frac{c}{C} \sum_{u_i \in F} m_U \left(B_i \cap O_\Omega\right) \geq \\
	& \geq \frac{c}{C} m_U\left((O_{\Omega} \cap O_{\tOmega_2}) \cap B^U_{L-Re^{T_2}}\right)
	\end{align*} 
	where the last inequality follows from the fact that $ \mathcal{F}' $ is a cover of the set $ (O_{\Omega} \cap O_{\tOmega_2}) \cap B^U_{L-Re^{T_2}} $.
	Making use of assumption (\ref{tech_ineq_x0_Chacon-Ornstein}) which states
	\[ m_U\left((O_{\Omega} \cap O_{\tOmega_2}) \cap B^U_{L-Re^{T_2}}\right) > \frac{1}{2} m_U(O_{\Omega} \cap O_{\tOmega_2}) \]
	we conclude
	\[ m_U\left( \bigcup_{u \in O_{\tOmega_2}} J_u \cap B^U_L \right) > \frac{c}{2C} m_U(O_{\Omega} \cap O_{\tOmega_2}), \]
	proving the claim with $ C_0=\frac{c}{2C} $. 
\end{proof}

\vspace{0.3cm}
Joining the results of lemmata \ref{lma_case1}--\ref{lma_cover} we deduce
\begin{eqnarray*}
	m_U (O_{\tOmega} \smallsetminus O_{\tOmega_2}) \geq m_U\left( \bigcup_{u \in O_{\tOmega_2}} J_u \cap B_L^U \right)
	\geq C_0 \cdot m_U(O_{\Omega} \cap O_{\tOmega_2})
\end{eqnarray*}
thus concluding the proof of \cref{main_theorem}.

\section{Consequences of the Main Theorem}\label{Section_corollaries}

\subsection{MA-quasi-invariance} We shall now prove our first corollary of the main theorem, giving a sufficient condition for $ MA $-quasi-invariance in the context of a general discrete group $ \Gamma $:

\vspace{2mm}
\begin{cornm}[\ref{cor_MA-qi}]
	Let $ \Gamma $ be any discrete subgroup of $ G $ and let $ \mu $ be any $ U $-e.i.r.m.~on $ \GmodGamma $. If $ \asymneighborhood $ contains a Zariski dense subgroup of $ G $ then $ \mu $ is $ MA $-quasi-invariant.	
\end{cornm}

\medskip

The corollary above for $ d=2 $ includes the case of weakly-tame surfaces introduced by Sarig in \cite{Sarig2010}. Sarig defined weakly-tame surfaces as hyperbolic surfaces satisfying the condition that all geodesic rays $ \{a_{-t}g\Gamma \}_{t \geq 0} $ intersect  an infinite number of times pairs of pants with bounded ``norm'' (sum of the lengths of closed geodesics bounding the pair of pants). Sarig's condition implies the existence of a sequence $ t_n \to \infty $ for which $ a_{-t_n}g\Gamma g^{-1}a_{t_n} $ contains a free group generated by two hyperbolic elements $ \gamma_1 $ and $ \gamma_2 $, with uniformly separated fixed points on $ \partial \mathbb{H}^2 $. This in turn implies that $ \mathcal{S}_{g\Gamma} $ contains a non-elementary and Zariski dense subgroup.

\begin{proof}[Proof of \cref{cor_MA-qi}]
	Let $ H=\overline{\langle \alpha(\asymneighborhood) \rangle} $ be the metric closure of the group generated by $ \alpha(\asymneighborhood) $ in $ MA $. By \cref{main_theorem}, $ H \subseteq H_\mu $. We want to show that $ MA \subseteq H $.
	
	Guivarc'h and Raugi \cite[Theorem 1]{Guivarc'h-Raugi} proved a related (more general) result, which in our setting translates to: 
	\[ \text{\emph{The group }} [M,M]H \text{\emph{ is a finite index subgroup of }} MA, \] 
	where $ [M,M] $ is the commutator subgroup of $ M $. Connectedness of $ MA $ further implies
	\begin{equation}\label{eq_GuivarchRaugi}
	[M,M]H = MA.
	\end{equation}
	In the case of $ d = 2 $ and $ 3 $ this completes the proof, as $ [M,M] = \{e\} $. 
	\medskip
	
	Assume $ d \geq 4 $. Equation \eqref{eq_GuivarchRaugi} implies in particular that $ MH = MA $, hence it would suffice to show that $ M \subseteq H $. To this end we show that the commutator subgroup of $ H $ admits $ [H,H] = M $.
	
	Recall equation (\ref{tech_alpha_conjugacyrelation}) in \S\ref{Section_reparam_map} stating that for any hyperbolic $ \gamma $ with $ \alpha(\gamma) \neq e $ there exist $ n \in N $ and $ u \in U $ satisfying
	\[ \gamma = (nu)^{-1}\alpha(\gamma)(nu). \] 
	Suppose $ \alpha(\asymneighborhood) $ is not Zariski dense in the group $ MA $ (considered as a real algebraic group). Then there is a subvariety $V \subset MA$ of strictly lower dimension containing $\alpha(\asymneighborhood)$.
	
	The map $ \Psi : N\times MA \times U \to G $ defined by $ (n,a^\prime,u) \mapsto (nu)^{-1} a^\prime (nu) $ is algebraic, and is a dominant morphism (the image contains a Zariski open dense set).
	Under the above non-density assumption $ \asymneighborhood \cap \mathrm{Image}(\Psi)$ is contained in $\Psi (N \times V \times U)$. Since  $\Psi (N \times V \times U)$ is a constructible subset of $G$ of lower dimension, it follows that $\asymneighborhood$ is contained in a proper subvariety of $G$, a contradiction.
	
	Therefore $ \alpha(\asymneighborhood) $ is Zariski dense in the group $ MA $.
	Consequently, $ [H,H] $ is Zariski dense in $ [MA,MA] = [M,M]$ (considered as an algebraic group). The group $ M $ is isomorphic as a Lie group to the simple Lie group $\SO(d-1)$, hence $[M,M]=M$ (also as Lie groups) and moreover any closed proper subgroup of $M$ is contained in a proper algebraic subgroup of $M$. It follows that $ [H,H]=[M,M] =M$, concluding the proof.
\end{proof}

One may deduce strong quasi-invariance implications also under weaker conditions on $ \asymneighborhood $, namely whenever $ \asymneighborhood $ contains a non-elementary subgroup (one which is not necessarily Zariski dense when $ d \geq 3 $):
\begin{cor}
	Let $ \Gamma $ be any discrete subgroup of $ G $ and let $ \mu $ be any $ U $-e.i.r.m.~on $ \GmodGamma $. If $ \asymneighborhood $ contains a non-elementary subgroup of $ G $, then the projected measure $ (\pi_M)_* \mu $ on $ M\backslash \GmodGamma \cong T^1(\Hd / \Gamma) $ is $ A $-quasi-invariant. 
\end{cor}

\begin{proof}
	As before, set $ H=\overline{\langle \alpha(\asymneighborhood) \rangle} $. For any hyperbolic element $ \gamma \in G $ let 
	\[ \ell(\gamma) = \inf_{z \in \Hd} d_{\Hd}(z,\gamma.z)  \]
	denote the translation length of $ \gamma $. Recall that $ \ell(m a_t)=|t| $ for any $ ma_t \in MA $.
	
	Kim \cite{Kim2006}, generalizing a theorem by Dal'bo \cite{Dalbo1999} (see also \cite[Ch.III Thm.3.6]{Dalbo_Book}), proved that the length spectrum of any non-elementary subgroup of $ G $ is non-arithmetic. Therefore there exist $ \gamma_1 , \gamma_2 \in \asymneighborhood $ for which $ \langle \ell(\gamma_1) , \ell(\gamma_2) \rangle $ is a dense subgroup of $ \R $. We show in \S\ref{Section_reparam_map} that given any hyperbolic element $ \gamma $, amongst $ \alpha(\gamma),\alpha(\gamma^{-1}) $ at least one is not equal to $ e $. Hence without loss of generality we may assume $ \alpha(\gamma_i) \neq e $ for both $ i=1,2 $. Since $ \alpha(\gamma_i) $ and $ \gamma_i $ are conjugate, they have equal translation lengths and therefore
	\[ \langle \ell\left(\alpha(\gamma_1)\right),\ell\left(\alpha(\gamma_2)\right) \rangle \]
	is dense in $ \R $.  
	Therefore $ M \backslash M{\alpha(\asymneighborhood)} $ generates a dense subgroup in $ M \backslash MA $, implying $ M \backslash MH = M \backslash MA $. \cref{main_theorem} now concludes the claim.
\end{proof}

\subsection{Measure Decomposition}\label{Subsection_measuredecomp}

Let $ Q = UMA $ be a minimal parabolic. Denote by $ \pi_-$ the quotient map $ g \mapsto Qg $ onto $ Q \backslash G \cong S^{d-1} $. The group $ G $ is diffeomorphic to
\[ U \times M \times A \times (Q \backslash G) \cong U \times M \times A \times S^{d-1} \]
The left action of $ UMA $ in these coordinates is easily computed as
\begin{equation}\label{tech_leftactUMAScoor}
uma.(u',m',a',\xi) = (u(u')^{ma},mm',aa',\xi)
\end{equation}
where $ \xi \in Q \backslash G $ and conjugation is denoted by $ g^h = hgh^{-1} $.

Fix $ \xi \in Q \backslash G $ and two points $ g,h \in G $. Denote by $ g_\xi $ an element in $ Kg $ satisfying $ \pi_-(g_\xi)=\xi $, respectively $ h_\xi \in Kh $ with $ \pi_-(h_\xi)=\xi $. Using the $ KAU $ decomposition we write
\[ hg_\xi^{-1} = k a_{b_\xi(g,h)} u. \]
The function $ b_\xi : G\times G \to \R $ above is called the Busemann function. This function is usually defined over $ \Hd \times \Hd $ and satisfies
\[ |b_\xi(g,h)| = \lim_{t \to \infty} d_{\Hd}(Ka_{-t}g_\xi, Ka_{-t}h_\xi). \]

Given the point $ p_\xi=(e,e,e,\xi) $ one can verify that the right action of an element $ \gamma \in \Gamma $ is given by
\[ (e,e,e,\xi).\gamma = (u({p_\xi \gamma}),m({p_\xi \gamma}),a_{b_\xi(\gamma^{-1},e)},\xi \gamma^{-1}) \]
for some $ u({p_\xi \gamma}) \in U $ and $ m({p_\xi \gamma}) \in M $ depending on $ \xi $ and $ \gamma $. Therefore
\[ (u,m,a_t,\xi).\gamma = \left(u (u({p_\xi \gamma}))^{ma_t},mm({p_\xi \gamma}),a_{t+b_\xi(\gamma^{-1},e)},\xi \gamma^{-1}\right). \]

A measure $ \nu $ on $ S^{d-1} \cong Q \backslash G $ is called $ \Gamma $-conformal with parameter (or dimension) $ \beta $ if
\[ \frac{d\gamma_* \nu}{d\nu}(\xi) = e^{-\beta b_\xi(\gamma^{-1},e)} \]
for every $ \gamma \in \Gamma $ and $ \xi \in S^{d-1} $, where $ \gamma $ acts on $ Q \backslash G $ on the right.

\begin{lemma}
	Let $ \mu $ be a $ U $-ergodic and invariant and $ MA $-quasi-invariant Radon measure on $ \GmodGamma $. Then $ \mu $ is induced by a measure $ \tmu $ on $ G $ presented in $ U \times M \times A \times S^{d-1}  $ coordinates as
	\[ d\tmu = e^{\beta t} d\nu dm dt du  \] 
	where $ \nu $ is a $ \Gamma $-conformal measure with parameter $ \beta $ on $ S^{d-1} $.
\end{lemma}

For the case of $d=2$ this correspondence was observed already in Burger's paper \cite[\S2.1]{Burger1990}, and has been highlighted by Babillot in her work, e.g.~\cite{Babillot2004}; cf.~also \cite{Ledrappier-Sarig} by Ledrappier and Sarig. Note that this correspondence also works the opposite way: starting from a $\Gamma$-conformal probability measure on $S^{d-1}$ one can use the above recipe to obtain a $U$-invariant Radon measure on $G/\Gamma$.

\begin{proof}
	
	Recall that whenever an ergodic measure $ \mu $ is quasi-invariant with respect to a group $ H $ then there exists some character $ \chi : H \to \R_{>0} $ satisfying for all $ h \in H $
	\[ h.\mu = \chi(h) \mu. \]
	Since $ M $ is compact it has no non-trivial continuous characters into $ \R $ implying $ \mu $ is $ M $-invariant. On the other hand, $ A $-quasi-invariance implies there exists $ \alpha \in \R $ for which
	\[ a_t.\mu = e^{\alpha t} \mu. \]
	Let $ \tmu $ be the lift of $ \mu $ to a $ \Gamma $-invariant measure on $ G $. Therefore $ \tmu $ may be presented in $ U \times M \times A \times S^{d-1}  $ coordinates as
	\[ d\tmu = e^{\beta t}d\nu dm dt du \]
	where $ \nu $ is some finite measure on $ S^{d-1} $. Note that $ \beta = \alpha - (d-1) $ since (\ref{tech_leftactUMAScoor}) implies that applying $ a_t $ on the left induces a change of variable $ u'=u^{a_t} $ with $ \frac{du'}{du}=e^{(d-1)t} $.
	
	\medskip
	We will show that $ \Gamma $-invariance of $ \tmu $ implies that $ \nu $ is $ \Gamma $-conformal with parameter $ \beta $. Fix $ \gamma \in \Gamma $ and denote $ B_\xi = b_\xi(\gamma^{-1},e) $. Given $ f \in C_c(G) $ we compute
	\begin{align} \label{eq:measure decomposition equation}
	\int f d\tmu &= \int f d\gamma_*\tmu = \int f\left((u,m,a_t,\xi).\gamma\right) d\tmu =  \\
	&= \int f\left(u (u({p_\xi \gamma}))^{ma_t},mm({p_\xi \gamma}),a_{t+B_\xi},\xi \gamma^{-1}\right)d\tmu \nonumber
	\end{align}
	note that $ u({p_\xi \gamma}) $ is independent of $ u $ and $ m({p_\xi \gamma}) $ independent of $ m $, therefore by Fubini's theorem we get
	\begin{align*}
	\eqref{eq:measure decomposition equation}&= \int f\left(u ,m,a_{t+B_\xi},\xi \gamma^{-1}\right)d\tmu = \\
	&= \int (f\circ a_{B_\xi})\left(u^{a_{-B_\xi}} ,m,a_{t},\xi \gamma^{-1}\right) d\tmu = \\
	&= \int f\left(u^{a_{-B_\xi}} ,m,a_{t},\xi \gamma^{-1}\right) d(a_{B_\xi}.\tmu) = \\
	&= \int f\left(u^{a_{-B_\xi}} ,m,a_{t},\xi \gamma^{-1}\right) e^{\alpha B_\xi} d\tmu 
\intertext{using the change of variable $ u'=u^{a_{-B_\xi}} $ with $ du'=e^{(d-1)B_\xi}du $ gives}
	\eqref{eq:measure decomposition equation} &= \int f\left(u',m,a_{t},\xi \gamma^{-1}\right) e^{(\alpha-(d-1))B_\xi+\beta t}d\nu dm dt du' =\\
	&= \int f\left(u',m,a_{t},\xi \right) e^{\beta (B_\xi + t)}d\gamma_*\nu dm dt du'.
	\end{align*}
	Consequently
	\[ \frac{d\gamma_* \nu}{d\nu}(\xi)=e^{-\beta b_\xi(\gamma^{-1},e)} \]
	as required.
\end{proof}

\subsection{Regular covers of Geometrically Finite Manifolds}\label{sec:regular cover}

Recall that the limit set $ \Lambda_\Gamma $ associated to a discrete group $ \Gamma < G $ is defined as the set of accumulation points of $ g\Gamma $ in $ \partial \Hd $ for some (and hence any) $ g \in G $. An equivalent definition for the limit set of $ \Gamma $ is the unique minimal closed $ \Gamma $-invariant subset of $ \partial\Hd $.

A limit point $ \xi \in \Lambda_\Gamma $ is called conical if for some (hence every) $ g \in \pi^{-1}(\xi) $ the geodesic ray $ \{a_{-t} g \Gamma \}_{t \geq 0} $ in $ \GmodGamma $ returns infinitely often within a bounded distance of $ g\Gamma $. 

A limit point $ \xi $ is called parabolic if it is fixed by a parabolic element of $ \Gamma $, and is called bounded parabolic if additionally
\[ (\Lambda_\Gamma \smallsetminus \{ \xi \}) / \Gamma_\xi \]
is compact, where $ \Gamma_\xi = \mathrm{Stab}_\Gamma(\xi) $ is the stabilizer of $ \xi $ in $ \Gamma $.

\begin{dfn}\label{def_GF}
	A group $ \Gamma_0 < G $ is called geometrically finite if every point in $ \Lambda_{\Gamma_0} $ is either conical or bounded parabolic.
\end{dfn}

See \cite[Theorem 12.4.5]{Ratcliffe_Book} and \cite{Bowditch} for other equivalent definitions of geometrical finiteness.
For $ d=2,3 $ this definition is also equivalent to $ \Gamma_0 $ having a finite-sided fundamental polyhedron in $ \Hd $. Both lattices and convex cocompact subgroups are examples of geometrically finite discrete subgroups.

Note that whenever $ \{e\} \neq \Gamma \lhd \Gamma_0 $ the space $ \GmodGamma $ is a regular cover of the geometrically finite hyperbolic manifold $ \GmodGamma_0 $. We prove the following:
\medskip

\begin{cornm}[\ref{cor_GF_intro}]
	Let $ \Gamma_0  $ be a geometrically finite Zariski dense discrete subgroup of $ G $ and let $ \{e\} \neq \Gamma \lhd \Gamma_0 $. Let $ \mu $ be a $ U $-e.i.r.m.~on $ \GmodGamma $. Then one of the following holds:
	\begin{enumerate}[leftmargin=*]
		\item $ \mu $ is supported on a wandering horosphere.
		\item $ \mu $ is supported on a lift to $ \GmodGamma $ of a horosphere in $ \GmodGamma_0 $ bounding a cusp.
		\item $ \mu $ is $ MA $-quasi-invariant.
	\end{enumerate}
\end{cornm}

In the special case where $\Gamma$ itself is geometrically finite one obtains (using Sullivan's theorem that there is only one atom-free $ \Gamma $-conformal measure supported on the limit set \cite{Sullivan1984}) that the Burger-Roblin measure on $ \GmodGamma $ for geometrically finite Zariski dense subgroups is the unique $U$-invariant non-periodic recurrent measure on $ \GmodGamma $. This case was shown by Winter \cite{Winter2015} (strengthening the analogous statement for $M\backslash\GmodGamma$ which was proved by Roblin~\cite{Roblin2003}; cf. also the proof by Schapira \cite{Schapira} as well as \cite[Thm.~5.7]{Mohammadi-Oh} by Mohammadi and Oh for another proof of this measure classification in $\GmodGamma$ in the geometrically finite context).

\begin{proof}[Proof of \cref{cor_GF_intro}]
	Denote by $ \Lambda_{\Gamma_0}^{par} $ and $ \Lambda_{\Gamma_0}^{con} $ the parabolic and conical limit points of $ \Lambda_{\Gamma_0} $. Note that $ \Lambda_{\Gamma_0}^{par} $ is countable. By \cref{def_GF}
	\begin{equation*}
	\Lambda_{\Gamma_0} = \Lambda_{\Gamma_0}^{par} \cup \Lambda_{\Gamma_0}^{con}. 
	\end{equation*}
	
	It is well known that whenever $ \{e\} \neq \Gamma \lhd \Gamma_0 $ then $ \Lambda_\Gamma = \Lambda_{\Gamma_0} $ (see e.g.~\cite[Theorem 12.2.14]{Ratcliffe_Book}). 
	We may thus conclude
	\begin{equation}\label{tech_limitptRCGF_classification}
	\Lambda_{\Gamma} = \Lambda_{\Gamma_0}^{par} \cup \Lambda_{\Gamma_0}^{con}. 
	\end{equation}
	
	The group $ G $ may be identified with the frame bundle of hyperbolic $ d $-space. Let $ \pi_- : G \to \partial \Hd $ denote the projection onto the boundary of hyperbolic space, where $ \pi_-(g) $ is defined as the limit of the geodesic trajectory $ \{a_{-t}g\}_{t \geq 0} $. 
	Note that $ \pi_-(ug)=\pi_-(g) $ for any $ g \in G $ and $ u \in U $. By (\ref{tech_limitptRCGF_classification}) we know
	\[ \partial\Hd = (\partial\Hd \smallsetminus \Lambda_\Gamma) \cupdot \Lambda_{\Gamma_0}^{par} \cupdot \Lambda_{\Gamma_0}^{con} \]
	This partition of $ \partial\Hd $ induces via $ \pi_{-}^{-1} $ a partition of $ G $ into $ U $-invariant sets. Ergodicity implies either:
	\begin{enumerate}[leftmargin=*]
		\item $ \mu $ is supported inside $ \pi_{-}^{-1}(\partial\Hd \smallsetminus \Lambda_\Gamma) / \Gamma $
		\item $ \mu $ is supported inside $ \pi_{-}^{-1}(\Lambda_{\Gamma_0}^{par}) / \Gamma $
		\item $ \mu $ is supported inside $ \pi_{-}^{-1}(\Lambda_{\Gamma_0}^{con}) / \Gamma $
	\end{enumerate}
	These three cases correspond respectively to the ones presented in the theorem. In case (1) the measure $ \mu $ is supported on a set of horospheres wandering into a funnel. Ergodicity implies that $ \mu $ is supported on a single wandering horosphere.
	
	For case (2), by ergodicity (and the fact that $\Lambda_{\Gamma_0}^{par}$ is a countable set) there is a single $ \Gamma $-orbit $ \xi.\Gamma \in \Lambda_{\Gamma_0}^{par} / \Gamma $ so that $\mu$ is supported inside $ \pi_{-}^{-1}(\xi.\Gamma) / \Gamma $. Applying ergodicity once again we see that $ p_*\mu $ must give full measure to a single horosphere bounding the cusp around $ \xi.\Gamma_0 $.
	
	Assume case (3), then for $ \mu $-a.e $ g\Gamma $, the projected geodesic ray 
	\[ p\left(\{a_{-t}g\Gamma\}_{t\geq 0}\right) = \{a_{-t}g\Gamma_0\}_{t\geq 0} \] 
	returns infinitely often to a compact set in $ \GmodGamma_0 $. Hence there exist sequences $ \gamma_n \in \Gamma_0 $ and $ t_n \to \infty $ for which $ a_{-t_n}g\gamma_n $ converges to some $ g_0 \in G $. Using the fact that $ \Gamma \lhd \Gamma_0 $ we deduce
	\[ a_{-t_n}g\Gamma g^{-1}a_{t_n}=a_{-t_n}g\left(\gamma_n \Gamma \gamma_n^{-1}\right) g^{-1}a_{t_n} \to g_0\Gamma g_0^{-1} \]
	and therefore 
	\[ g_0\Gamma g_0^{-1} \subseteq \asymneighborhood. \] 
	
	Note that $ Z\mbox{-}cl(\Gamma) $, the Zariski closure of $ \Gamma $, is invariant under conjugation by $ G=Z\mbox{-}cl(\Gamma_0) $. Therefore since $ G $ is a simple Lie group we deduce $ \Gamma $ and also $ g_0\Gamma g_0^{-1} $ are Zariski dense. \cref{cor_MA-qi} implies the measure $ \mu $ is $ MA $-quasi-invariant.
\end{proof}

\subsection{Bounded Injectivity Radius}\label{Subsection_injrad}

The injectivity radius at a point $ x $ in $ \GmodGamma $ is defined as
\[ \injrad(x) = \sup \{r : v_1x \neq v_2x \;\; \forall v_1,v_2 \in B_r^G \}. \]

\begin{cornm}[\ref{cor_injrad_intro}]
	Let $ \Gamma < G $ be any discrete group with injectivity radius uniformly bounded away from 0, i.e.~$ \inf_{x \in \GmodGamma} \injrad(x) > 0 $. Let $ \mu $ be any $ U $-e.i.r.m. Then one of the following holds:
	\begin{enumerate}[leftmargin=*]
		\item $ \lim_{t \to \infty} \injrad(a_{-t} x) = \infty $ for $ \mu $-a.e.~$ x $.
		\item $ H_\mu $ is cocompact in $ MA $, i.e.~ $ H_\mu $ contains a non-trivial hyperbolic element.
	\end{enumerate}
\end{cornm}

\noindent
Note that the above conditions on $ \Gamma $ are equivalent to the group being purely-hyperbolic with length spectrum bounded away from 0, i.e.~$ \inf \ell(\Gamma) > 0 $.

\medskip

Whenever $ H_\mu $ is cocompact in $ MA $ there exists a character $ \chi : H_\mu \to \R $ satisfying
\[ \ell.\mu = \chi(\ell) \mu \quad \text{for all } \ell \in H_\mu. \]
This character may be uniquely extended to be defined on all of $ MA $. Indeed, $ \chi $ extends trivially to $ MH_\mu $ and extends linearly to $ MA $ using the fact that $ M \backslash MH_\mu $ is cocompact in $ M \backslash MA \cong \R $. The measure
\[ \lambda = \int_{{MA}/{H_\mu}} \chi(\ell^{-1})\ell.\mu\; dm_{{MA}/{H_\mu}}(\ell) \]
on $ \GmodGamma $ is Radon $ U $-invariant and $ MA $-quasi-invariant. Furthermore, by our construction $ ma.\lambda = \chi(ma) \lambda $ for any $ ma \in MA $. In other words, under these conditions the measure $ \mu $ is an ergodic component of a measure of the form $ e^{\beta t}d\nu dm dt du $ where $ \nu $ is a $ \Gamma $-conformal measure on $ S^{d-1} $.

\begin{proof}
	Whenever $ \injrad(g\Gamma)<R $ there exist $ v_1,v_2 \in B_R^G $ with $ v_1 \neq v_2 $ satisfying $ v_1 v_2^{-1} g\Gamma = g\Gamma $, implying
	\[ v_1 v_2^{-1} \in g\Gamma g^{-1} \cap B_{2R}^G. \]
	Note that the map
	\[ x \mapsto \liminf_{t \to \infty} \injrad(a_{-t} x) \]
	from $ \GmodGamma $ to $ (0,\infty] $ is $ U $-invariant and measurable and therefore constant $ \mu $-a.s.
	Assume that
	\[ \liminf_{t \to \infty} \injrad(a_{-t} x) < \infty \quad \mu \text{-a.s.}\]
	Then for $ \mu $-a.e.~$ g\Gamma \in \GmodGamma $ there exists $ R > 0 $ and sequences $ t_n \to \infty $ and $ \gamma_n \in G \smallsetminus \{e\} $ satisfying
	\[ \gamma_n \in a_{-t_n}g\Gamma g^{-1}a_{t_n} \cap B_{2R}^G. \]
	Let $ \gamma_0 \in B_{2R}^G $ be an accumulation point of the sequence $ \gamma_n $. Since the length of a hyperbolic element is preserved by conjugation, and since $ \inf \ell(\Gamma) = c > 0 $ we are ensured that $ \gamma_0 $ is hyperbolic with $ \ell(\gamma_0) \geq c $. Therefore $ \asymneighborhood $ contains a non-trivial hyperbolic element and \cref{main_theorem} implies $ H_\mu \smallsetminus M \neq \{e\} $, proving the claim.
\end{proof}


\begin{thebibliography}{24}
	
	\bibitem{Babillot2004}
	Martine Babillot.
	\newblock On the classification of invariant measures for horosphere foliations
	on nilpotent covers of negatively curved manifolds.
	\newblock In {\em Random walks and geometry}, pages 319--335. Walter de
	Gruyter, Berlin, 2004.
	
	\bibitem{Babillot-Ledrappier}
	Martine Babillot and Fran{\c{c}}ois Ledrappier.
	\newblock Geodesic paths and horocycle flow on Abelian covers.
	\newblock In {\em Lie groups and ergodic theory ({M}umbai, 1996)}, volume~14 of
	{\em Tata Inst. Fund. Res. Stud. Math.}, pages 1--32. Tata Inst. Fund. Res.,
	Bombay, 1998.
	
	\bibitem{Bowditch}
	B.~H. Bowditch.
	\newblock Geometrical finiteness for hyperbolic groups.
	\newblock {\em J. Funct. Anal.}, 113(2):245--317, 1993.
	
	\bibitem{Burger1990}
	Marc Burger.
	\newblock Horocycle flow on geometrically finite surfaces.
	\newblock {\em Duke Math. J.}, 61(3):779--803, 1990.
	
	
	\bibitem{Dalbo1999}
	Fran\c{c}oise Dal'bo.
	\newblock Remarques sur le spectre des longueurs d'une surface et comptages.
	\newblock {\em Bol. Soc. Brasil. Mat. (N.S.)}, 30(2):199--221, 1999.
	
	\bibitem{Dalbo_Book}
	Fran\c{c}oise Dal'Bo.
	\newblock {\em Geodesic and horocyclic trajectories}.
	\newblock Universitext. Springer-Verlag London, Ltd., London; EDP Sciences, Les
	Ulis, 2011.
	\newblock Translated from the 2007 French original.
	
	\bibitem{Dani1978}
	S.~G. Dani.
	\newblock Invariant measures of horospherical flows on noncompact homogeneous
	spaces.
	\newblock {\em Invent. Math.}, 47(2):101--138, 1978.
	
	\bibitem{Furstenberg1973}
	Harry Furstenberg.
	\newblock The unique ergodicity of the horocycle flow.
	\newblock In {\em Recent advances in topological dynamics ({P}roc. {C}onf.,
		{Y}ale {U}niv., {N}ew {H}aven, {C}onn., 1972; in honor of {G}ustav {A}rnold
		{H}edlund)}, pages 95--115. Lecture Notes in Math., Vol. 318. Springer,
	Berlin, 1973.
	
	\bibitem{Guivarc'h-Raugi}
	Y~Guivarch and A~Raugi.
	\newblock Actions of large semigroups and random walks on isometric extensions
	of boundaries.
	\newblock {\em Annales Scientifiques de l’École Normale Supérieure},
	40(2):209--249, March 2007.
	
	\bibitem{Hochman2010}
	Michael Hochman.
	\newblock A ratio ergodic theorem for multiparameter non-singular actions.
	\newblock {\em J. Eur. Math. Soc. (JEMS)}, 12(2):365--383, 2010.
	
	\bibitem{Kim2006}
	Inkang Kim.
	\newblock Length spectrum in rank one symmetric space is not arithmetic.
	\newblock {\em Proc. Amer. Math. Soc.}, 134(12):3691--3696, 2006.
	
	\bibitem{Ledrappier2008}
	Fran\c{c}ois Ledrappier.
	\newblock Invariant measures for the stable foliation on negatively curved
	periodic manifolds.
	\newblock {\em Ann. Inst. Fourier (Grenoble)}, 58(1):85--105, 2008.
	
	\bibitem{Ledrappier-Sarig}
	Fran{\c{c}}ois Ledrappier and Omri Sarig.
	\newblock Invariant measures for the horocycle flow on periodic hyperbolic
	surfaces.
	\newblock {\em Israel J. Math.}, 160:281--315, 2007.
	
	\bibitem{mattila1999geometry}
	P.~Mattila.
	\newblock {\em Geometry of Sets and Measures in Euclidean Spaces: Fractals and
		Rectifiability}.
	\newblock Cambridge Studies in Advanced Mathematics. Cambridge University
	Press, 1999.
	
	\bibitem{Maucourant-Schapira}
	Fran\c{c}ois Maucourant and Barbara Schapira.
	\newblock On topological and measurable dynamics of unipotent frame flows for
	hyperbolic manifolds.
	\newblock {\em Duke Math. J.}, 168(4):697--747, 2019.
	
	\bibitem{Mohammadi-Oh}
	Amir Mohammadi and Hee Oh.
	\newblock Classification of joinings for {K}leinian groups.
	\newblock {\em Duke Math. J.}, 165(11):2155--2223, 2016.
	
	\bibitem{Oh-Pan}
	Hee Oh and Wenyu Pan.
	\newblock Local mixing and invariant measures for horospherical subgroups on
	abelian covers.
	\newblock {\em arXiv preprint arXiv:1701.02772}, 2017.
	
	\bibitem{Ratcliffe_Book}
	John~G. Ratcliffe.
	\newblock {\em Foundations of hyperbolic manifolds}, volume 149 of {\em
		Graduate Texts in Mathematics}.
	\newblock Springer, New York, second edition, 2006.
	
	\bibitem{Ratner1991}
	Marina Ratner.
	\newblock On {R}aghunathan's measure conjecture.
	\newblock {\em Ann. of Math. (2)}, 134(3):545--607, 1991.
	
	\bibitem{Ratner1992_SL2(R)}
	Marina Ratner.
	\newblock Raghunathan's conjectures for {${\rm SL}(2,\mathbb{R})$}.
	\newblock {\em Israel J. Math.}, 80(1-2):1--31, 1992.
	
	\bibitem{Roblin2003}
	Thomas Roblin.
	\newblock Ergodicit\'e et \'equidistribution en courbure n\'egative.
	\newblock {\em M\'em. Soc. Math. Fr. (N.S.)}, .(95):vi+96, 2003.
	
	\bibitem{Sarig2004}
	Omri Sarig.
	\newblock Invariant {R}adon measures for horocycle flows on abelian covers.
	\newblock {\em Invent. Math.}, 157(3):519--551, 2004.
	
	\bibitem{Sarig2010}
	Omri Sarig.
	\newblock The horocyclic flow and the {L}aplacian on hyperbolic surfaces of infinite genus.
	\newblock {\em Geom. Funct. Anal.}, 19(6):1757--1812, 2010.
	
	\bibitem{Schapira}
	Barbara Schapira.
	\newblock A short proof of unique ergodicity of horospherical foliations on
	infinite volume hyperbolic manifolds.
	\newblock {\em Confluentes Math.}, 8(1):165--174, 2016.
	
	\bibitem{Sullivan1984}
	Dennis Sullivan.
	\newblock Entropy, hausdorff measures old and new, and limit sets of
	geometrically finite kleinian groups.
	\newblock {\em Acta Math.}, 153:259--277, 1984.
	
	\bibitem{Veech1977}
	William~A. Veech.
	\newblock Unique ergodicity of horospherical flows.
	\newblock {\em Amer. J. Math.}, 99(4):827--859, 1977.
	
	\bibitem{Winter2015}
	Dale Winter.
	\newblock Mixing of frame flow for rank one locally symmetric spaces and
	measure classification.
	\newblock {\em Israel J. Math.}, 210(1):467--507, 2015.
	
\end{thebibliography}
\end{document}